\documentclass[final]{siamltex}

\usepackage{amssymb,latexsym,amsmath}
\usepackage{epsfig}
\usepackage{caption}
 \usepackage{epsfig}

\usepackage{makeidx}

\allowdisplaybreaks

 \usepackage{graphicx,fancybox,latexsym,epsfig}
\usepackage{fancyhdr,amsmath,times,amsxtra,amssymb}
\usepackage{color}

\usepackage{fancybox}
\usepackage{tikz}

\usepackage{amsmath}
\usepackage{amsfonts}
\usepackage{array}
\usepackage{dsfont}
\usepackage{hyperref}
\usepackage{amssymb}
\usepackage{bbold}
\usepackage{standalone}
\usepackage{accents}
\usepackage{enumitem}
\usepackage{tikz}
\usepackage{pgfplots}
\pgfplotsset{compat=1.6}
\usepgfplotslibrary{groupplots}
\usepackage{cancel}
 \usepackage{color}

\allowdisplaybreaks

\newtheorem{assumption}{Assumption}[section]

  \newtheorem{remark}{Remark}

\usepackage{amssymb,latexsym,amsmath}
\usepackage{mathrsfs}
\usepackage{epsfig}
\usepackage{caption}

\newcommand{\indep}{\raisebox{0.05em}{\rotatebox[origin=c]{90}{$\models$}}}
\newcommand\scalemath[2]{\scalebox{#1}{\mbox{\ensuremath{\displaystyle #2}}}}

\pgfplotsset{soldot/.style={color=blue,only marks,mark=*}}
\pgfplotsset{holdot/.style={color=blue,fill=white,only marks,mark=*}}
\fussy

\allowdisplaybreaks

\begin{document}
\sloppy
\title{Optimal Policies for Convex Symmetric Stochastic Dynamic Teams and their Mean-field Limit
\thanks{This research was supported by the Natural Sciences and Engineering Research Council (NSERC) of Canada. A summary of some of the results here is included in \cite{sinacdc2019convex} presented at the 2019 IEEE Conference
on Decision and Control (CDC). The authors are with the Department of Mathematics and Statistics,
     Queen's University, Kingston, ON, Canada.}
     Email: \{16sss3,yuksel@queensu.ca\}}
\author{Sina Sanjari and Serdar Y\"uksel}
\maketitle

\begin{abstract}{
This paper studies convex stochastic dynamic team problems with finite and infinite time horizons under decentralized information structures. First, we introduce two notions called  exchangeable teams and symmetric information structures. We show that in convex exchangeable team problems an optimal policy exhibits a symmetry structure. We give a characterization for such symmetrically optimal teams for a general class of convex dynamic team problems under a mild conditional independence condition. In addition, through concentration of measure arguments, we establish the convergence of optimal policies for teams with $N$ decision makers to the corresponding optimal policies for symmetric mean-field teams with infinitely many decision makers.  As a by-product, we present an existence result for convex mean-field teams, where the main contribution of our paper is with respect to the information structure in the system when compared with the related results in the literature that have either assumed a classical information structure or a static information structure. We also apply these results to the important special case of Linear Quadratic Gaussian (LQG) team problems, where while for partially nested LQG team problems with finite time horizons it is known that the optimal policies are linear, for infinite horizon problems the linearity of optimal policies has not been established in full generality. We also study average cost finite and infinite horizon dynamic team problems with a symmetric partially nested information structure and obtain globally optimal solutions where we establish linearity of optimal policies.}
\end{abstract}

\begin{keywords}
Stochastic teams, average cost optimization, decentralized control, mean-field teams
\end{keywords}

%% ------------------------------------------------------------------
%% END HEADER
%% ------------------------------------------------------------------

\section{Introduction and literature review}
\label{sec:intro}

Team problems consist of a collection of decision makers or agents acting together to optimize a common cost function, but not necessarily sharing all the available information. The term \textit{stochastic teams} refers to the class of team problems where there exist randomness in the initial states, observations, cost realizations, or the evolution of the dynamics. At each time stage, each agent only has partial access to the global information which is defined by the \textit{information structure} (IS) of the problem  \cite{wit75}. If there is a pre-defined order in which the decision makers act then the team is called a \textit{sequential team}. For sequential teams, if each agent's information depends only on primitive random variables, the team is \textit{static}. If at least one agent's information is affected by an action of another agent, the team is said to be \textit{dynamic}. Information structures can be further categorized as \textit{classical}, \textit{partially nested}, and \textit{non-classical}. An IS is \textit{classical} if the information of decision maker $i$ (DM$^i$) contains all of the information available to DM$^k$ for $k < i$. An IS is \textit{partially nested}, if whenever the action of DM$^k$, for some $k < i$, affects the information of DM$^i$, then the information of DM$^i$ contains the information of DM$^k$. An IS which is not partially nested is \textit{non-classical}. A detailed review is presented in \cite{YukselBasarBook}.

{ 
Obtaining structural results in team problems is important towards establishing both existence and computational/approximation methods for optimal policies. In this paper, we define the notion of exchangeable teams and symmetric information structures, and we show that, for convex exchangeable dynamic teams with finite horizons, optimal policies exhibit a symmetry structure (Theorem \ref{the:3.1}). For any number of DMs, this symmetry structure is more relaxed when compared with the symmetry results developed earlier, e.g. in \cite{sanjari2018optimal,sinacdc2018optimal} which focused on problems under a static information structure, and is applicable for dynamic teams which may not admit a static reduction, as long as convexity in policies holds for the team problem.

There have been many studies involving decentralized stochastic control with infinitely many decision makers. In particular, when the coupling among the decision makers is only through some aggregate/average effect, such problems can be viewed within the umbrella of mean-field games \cite{LyonsMeanField,CainesMeanField1}, which were introduced as a limit model for non-cooperative symmetric $N$-player differential games with a mean-field interaction as $N\to \infty$. The solution concept in game theory is often Nash equilibrium, and often under various characterizations of it in dynamic Bayesian setups. In the context of decentralized stochastic control or teams, these would correspond to person-by-person optimal solutions, and hence not necessarily globally optimal solutions.

Nonetheless, on the existence as well as uniqueness and non-uniqueness results on equilibria, there have been several studies for mean-field games \cite{LyonsMeanField, bardi2017non, carmona2016mean, light2018mean, lacker2015mean, bayraktar2020non, hajek2019non}. There have also been several studies for mean-field games where the limits of sequences of Nash equilibria have been investigated as the number of decision makers $N \to \infty$ (see e.g., \cite{fischer2017connection, lacker2018convergence, bardi2014linear, LyonsMeanField, arapostathis2017solutions}). We refer interested readers to \cite{carmona2018probabilistic, caines2018peter} for a literature review and a detailed summary of some recent results on mean-field games.

Some notable relevant studies from the mean-field literature are the following: In \cite{fischer2017connection}, through a concentration of measures argument, it has been shown that sequences of $\epsilon_{N}$- local (for each player) Nash equilibria for $N$ player games converge to a solution for the mean-field game under exchangeability of the initial states and weak convergence of normalized occupational measures to a deterministic measure \cite[Theorem 5.1]{fischer2017connection}. In \cite{lacker2016general}, assumptions on equilibrium policies of the large population mean-field symmetric stochastic differential games have been presented to allow the convergence of asymmetric approximate Nash equilibria to a weak solution of the mean-field game \cite[Theorem 2.6]{lacker2016general}. 

However, in these studies the information structures are restricted to the following models: In \cite{fischer2017connection} the information structure is assumed to be static since strategies of each player are assumed to be adapted to the filtration generated by his/her initial states and Wiener process (also called \textit{distributed open-loop} controllers in the mean-field games' literature \cite{lacker2016general,carmona2016mean, carmona2018probabilistic}) (see Remark \ref{rem:8} for details of this discussion). Convergence of Nash equilibria induced by closed-loop controllers to a weak semi-Markov mean-field equilibrium has been established in \cite{lacker2018convergence} for finite horizon mean-field game problems, where the classical information structure (i.e., what would be a centralized problem in the team theoretic setup) has been considered. For infinite horizon problems, in \cite{cardaliaguet2019example}, an example of ergodic differential games with mean-field coupling has been constructed such that limits of sequences of expected costs induced by symmetric Nash-equilibrium policies of $N$-player games capture expected costs induced by many more Nash-equilibrium policies including a mean-field equilibrium and social optima. In \cite{lacker2018convergence}, the classical information structure (a centralized problem) has been considered, where in \cite{cardaliaguet2019example} it has been assumed that players have access to all the history of states of all players but not controls (we note that in the team problem setup through using a classical result of Blackwell \cite{Blackwell2} in the case where each DM knows all the history of states of all DMs, optimal policies can be realized as one in the centralized problem where just the global state is a sufficient statistic). Moreover, under relaxed regularity conditions on dynamics and the cost function, a limit theory has been established for controlled McKean-Vlasov dynamics \cite{lacker2017limit} under the classical information structure, where through a similar analysis as in \cite{fischer2017connection, lacker2016general}, it has been shown that the empirical measure of pairs of states and $\epsilon_{N}$-open-loop optimal controls converges weakly as $N \to \infty$ to limit points in the set of pairs of states and optimal controls of the McKean-Vlasov problem. 

The above highlights the intricacies due to the information structure aspects: different from the aforementioned studies above, we consider information structures that are not necessarily static or classical. Also, in this paper, we work with global optimality and not only mean-field equilibria and we show the existence of a globally optimal policy for mean-field team problems. On the other hand, in our paper since we work under the convexity assumption, the information structure does not allow for the mean-field coupling in the dynamics. We also note that in prior work, \cite{sanjari2018optimal}, we studied static teams where under convexity and more restrictive symmetry conditions, global optimality of a limit policy of a sequence of $N$-DM optimal policies has been established. 

In the context of stochastic teams with countably infinite number of decision makers, the gap between  person-by-person optimality (Nash equilibrium in the game-theoretic context) and global team optimality is significant since a perturbation of finitely many policies fails to deviate the value of the expected cost, thus  person-by-person optimality is a weak condition for such a setup, and hence the results presented in the aforementioned papers may be inconclusive regarding global optimality of the limit equilibrium.  
For teams and social optima control problems, the analysis has primarily focused on the LQG model or Markov chains where the centralized performance has been shown to be achieved asymptotically by decentralized controllers (see e.g., \cite{huang2012social, arabneydi2017certainty, arabneydi2015team}).

We also obtain existence results on optimal policies for the setups considered. Compared to the results on the existence of a globally optimal policy in team problems where (finite) $N$-DM team problems have been considered \cite{yuksel2018general, gupta2014existence, YukselSaldiSICON17, saldi2019topology}, we study convex team problems with countably infinite number of decision makers. 
}

Parts of our results in this paper correspond to LQG teams. In \cite{HoChu}, it has been shown that for teams with finite number of DMs, dynamic teams with a partially nested information structure can be reduced to a static one (\cite{HoChu, wit88}) where Radner's theorem concludes global optimality of linear policies for LQG team problems \cite{Radner}. However, for average cost infinite horizon, partially nested, LQG dynamic team problems so far there has been no universal result establishing that a globally optimal policy is linear, time-invariant, and stabilizing, and this has been often imposed apriori: In \cite{Rotkowitz}, the problem of designing a linear, time-invariant, stabilizing, state feedback optimal policy for decentralized $\mathcal{H}_{2}$-optimization problems, which satisfy the quadratically invariance property, has been addressed by reparametrizing the problem as a convex problem (via Youla parameterization). In \cite{rotkowitz2008information}, it has been shown that for sequential team problems involving linear systems, quadratic invariance and the partially nested property are equivalent. For a class of partially ordered (POSET) systems, state space techniques have been utilized to obtain optimal, linear, time-invariant, state feedback controllers for $\mathcal{H}_{2}$-optimization problems with sparsity constraints \cite{shah2013cal}. A similar result has been established in \cite{swigart2014optimal} where linearity and time invariance have been imposed apriori.  In \cite{lessard2015optimal}, $\mathcal{H}_{2}$-optimization output feedback problems with two-players have been considered and optimality results have been established when the optimal policies are restricted to linear, time invariant, stabilizing policies. However, {the results in \cite{lessard2015optimal, Rotkowitz, shah2013cal, swigart2014optimal} are inconclusive regarding global optimality. Our contribution here is to consider average cost infinite horizon dynamic team problems without restricting the set of policies to those that are linear, time-invariant, and stabilizing unlike the results in \cite{lessard2015optimal, Rotkowitz, shah2013cal, swigart2014optimal}. We note again that the optimality of linear policies for infinite horizon LQG problems is an open problem in its generality and we provide positive results for a class of such problems.} 

{\bf Contributions.} 
In view of the discussion above, our paper makes the following contributions.
\begin{itemize}
\item[(i)] We define a notion of exchangeable teams and symmetric information structures, and we show that, for convex exchangeable dynamic teams with finite horizons, optimal policies exhibit a symmetry structure (Theorem \ref{the:3.1}). For any number of DMs, this symmetry structure is more relaxed when compared with the symmetry results developed in \cite{sanjari2018optimal, sinacdc2018optimal} and is applicable for dynamic teams which may not admit a static reduction, as long as convexity in policies holds for the team problem.
\item[(ii)] For convex mean-field teams with a symmetric information structure, through concentration of measure arguments, we establish the convergence of optimal policies for mean-field teams with $N$ decision makers to the corresponding optimal policies for mean-field teams (see Theorem \ref{the:mft}).
\item[(iii)] We establish an existence result for the class of convex mean-field teams with a symmetric information structure (see Theorem \ref{the:exmft}) for finite horizon problems, where, as noted in the literature review, related results assumed more restrictive information structures which are either static or classical.
\item [(iv)] We also apply our results to LQG dynamic teams for finite horizon problems (see Section \ref{sec:LQG}). For LQG dynamic teams with a symmetric partially information pattern, we obtain an optimal policy for finite horizon problems (see Section \ref{sec:3a}). We also apply convex mean-field results to LQG mean-field teams with a symmetric partially nested information structure (see Section \ref{sec:3a}) and obtain a globally optimal policy. Building on the result above, we also obtain a globally optimal policy for average cost LQG team problems. 

\end{itemize}

The organization of the paper is as follows: we study convex exchangeable dynamic teams with finite horizons in Section \ref{sec:4}, and we study mean-field teams in Section \ref{sec:mft}. We obtain globally optimal solutions for finite horizon problems with a symmetric partially nested information structure and LQG mean-field  teams in Section \ref{sec:3a}, and we discuss average cost LQG team problems with a symmetric information structure in Section \ref{sec:ave}, respectively.

{\bf Notation.} $\mathbb{R}$ and $\mathbb{N}$ denote the set of real numbers and natural numbers, respectively. We denote trace of a matrix $A$ as $Tr(A)$. We denote that a random vector $X$ is independent of a random vector $Y$ by $X \indep Y$. We denote $A^{T}$ as the transpose of a matrix $A$ and $A^{(T)}$ to show the dependence of a matrix $A$ to $T \in \mathbb{N}$. For any random variables $z^{1:N}:=(z^{1},\dots,z^{N})$, we defined ${z}^{-i}:=({z}^{1},\dots,{z}^{i-1},{z}^{i+1},\dots, {z}^{N})$, and $\mathcal{M}_{r,q}$ denotes the space of $r \times q$ matrices. 

\subsection{Preliminaries}\label{sec:pre}

In this section, we introduce Witsenhausen's \textit{Intrinsic Model} for sequential teams \cite{wit75} (we generalize this definition to infinite number of decision makers).
Consider \textit{sequential systems} and assume the action and measurement spaces are standard Borel spaces, that is, Borel subsets of complete, separable and metric spaces. The \textit{Intrinsic Model} for sequential teams is defined as follows.
\begin{itemize}
\item There exists a collection of \textit{measurable spaces} $\{(\Omega, {\cal F}), \allowbreak(\mathbb{U}^i,{\cal U}^i), (\mathbb{Y}^i,{\cal Y}^i), i \in {\mathcal{N}}\}$, specifying the system's distinguishable events, and control and measurement spaces. The set $\mathcal{N}$ denotes the collection of decision makers. The set $\mathcal{N}$ can be a finite set $\{1,2,\dots, N\}$ or a countable set $\mathbb{N}$. The pair $(\Omega, {\cal F})$ is a
measurable space (on which an underlying probability may be defined). The pair $(\mathbb{U}^i, {\cal U}^i)$
denotes the Borel space from which the action $u^i$ of DM$^i$ is selected. The pair $(\mathbb{Y}^i,{\cal Y}^i)$ denotes the Borel observation/measurement space.

\item There is a \textit{measurement constraint} to establish the connection between the observation variables and the system's distinguishable events. The $\mathbb{Y}^i$-valued observation variables are given by $y^i=h^i(\omega,{\underline u}^{1:i-1})$, where ${\underline u}^{1:i-1}=\{u^k, k \leq i-1\}$ and $h^i$s are measurable functions.
\item The set of admissible control laws $\underline{\gamma}= \{\gamma^i\}_{i \in \mathcal{N}}$, also called
{\textit{designs}} or {\textit{policies}}, are measurable control functions, so that $u^i = \gamma^i(y^i)$. Let $\Gamma^i$ denote the set of all admissible policies for DM$^i$ and let ${\Gamma} = \prod_{i \in \mathcal{N}} \Gamma^i$.
\item There is a {\textit{probability measure}} ${P}$ on $(\Omega, {\cal F})$ describing the probability space on which the system is defined.
\end{itemize}

Under the intrinsic model, every DM acts separately. However, depending on the information structure, it may be convenient to consider a collection of DMs as a single DM acting at different time instances. In fact, in the classical stochastic control, this is the standard approach.  

\section{Finite horizon convex dynamic team problems with a symmetric information structure}\label{sec:4}
In this section, we characterize symmetry in dynamic team problems. According to the discussion above, by considering a collection of DMs as a single DM ($i=1,\dots,N$) acting at different time instances ($t=0,\dots,T-1$), we define a team problem with $(NT)$-DMs as a team with $N$-DMs:
\begin{itemize}[wide]
\item  [{(i)}] Let the observation and action spaces be Borel subsets of $\mathbb{R}^{n}$ for a positive integer $n$ and be identical for each DM ($i=1,\dots,N$) with $\mathbb{Y}_{i}:={\bf{Y}}=\prod_{t=0}^{T-1}\mathbb{Y}^{t}$, $\mathbb{U}_{i}:={\bf{U}}=\prod_{t=0}^{T-1}\mathbb{U}^{t}$, respectively. The sets of all admissible policies are denoted by ${\bf{\Gamma}} = \prod_{i=1}^{N}\Gamma_{i}=\prod_{i=1}^{N}\prod_{t=0}^{T-1} \Gamma^{t}$.
\item [(ii)] For $i=1,\dots,N$, $y^{i}_{t}:=h_{t}^{i}(x_{0}^{1:N}, \zeta^{1:N}_{0:t}, u_{0:t-1}^{1:N})$ represents the observation of DM$^i$ at time $t$ ($h_{t}^{i}$s are Borel measurable functions).
\item [{(iii)}] Let $(\underline{\zeta}^{1:N}):=(\underline{\zeta}^{1},\dots,\underline{\zeta}^{N})$ where $\underline\zeta^{i}:=(x_{0}^{i}, \zeta_{0:T-1}^{i})$ denotes all the uncertainty associated with DM$^i$ including his/her initial states. We assume that $(\underline{\zeta}^{i})$ takes values in $\Omega_{\zeta}$ (where at each time instances $t$, it takes value in $\Omega_{\zeta_{t}}$). Let $\mu$ denote the law of $\underline{\zeta}^{1:N}$.
\item [{(iv)}] Define the expected cost function of $\underline{\gamma}^{1:N}$ as 
$J_{N}(\underline{\gamma}^{1:N})= E^{\underline{\gamma}^{1:N}}[c(\underline{\zeta}^{1:N},\underline{u}^{1:N})]$,
for some Borel measurable cost function $c: \prod_{i=1}^{N} ({\Omega_{\zeta}}\times{\bf{U}}) \to \mathbb{R}_{+}$, where $\underline{\gamma}^{1:N}=(\underline{\gamma}^{1},\underline{\gamma}^{2},\dots,\underline{\gamma}^{N})$ and $\underline{\gamma}^{i}=\gamma^{i}_{0:T-1}$ for $i=1,\dots,N$.
\end{itemize}

Now, we present the definition of symmetric information structures (note that symmetric information structures can be classical, partially nested, or non-classical).
\begin{definition}\label{def:s}
Let the information of DM$^{i}$ acting at time $t$ be described as $I_{t}^{i}:=\{y_{t}^{i}\}$. The information structure of a sequential $N$-DM team problem is \textit{symmetric} if  
\begin{itemize}[wide=0pt]
\item[(i)] $y^{i}_{t}=h_{t}(x_{0}^{i},x_{0}^{-i}, \zeta^{i}_{0:t}, {\zeta}^{-i}_{0:t}, u_{0:t-1}^{i}, u_{0:t-1}^{-i})$ where $h_{t}$ is identical for all $i=1,\dots,N$ (note that function's arguments depend on $i$).
\end{itemize}
\end{definition}

We note that the above definition can be generalized to be applicable for teams with countably infinite DMs and infinite horizon problems.

The symmetric information structure can also be interpreted and defined as a graph, which has often been the common method to describe information structures in control theory, relating DMs and their information through directed edges. Consider $G(V,\mu)$ as a directed graph with $V=\{1,\dots,NT\}$ nodes and where $\mu \subset V \times V$ determines the directed edges between nodes; this represents the dependency notation in the information of nodes, i.e., $(i,j)$ denotes a directed edge from $i$ to $j$, $i \to j$, it means $u^{i}$ affects $y^{j}$ through the relation $y^i=h^i(\omega,{\underline u}^{1:i-1})$ defined in the intrinsic model (see Section \ref{sec:pre}). We denote by $\downarrow j$ as the set of nodes $i$ such that $i \to j$ (ancestors), and $\downarrow\downarrow j=\{\downarrow j\}\cup\{j\}$. Similarly, we can define descendants by $\uparrow j$. We can define a collection of DMs as a single DM ($i=1,\dots,N$) acting at different time instances ($t=0,\dots,T-1$) on a graph with a symmetric information structure (two examples are shown in Fig. 2.1, and Fig. 2.2). Assume
\begin{itemize}[wide = 0pt]
\item [{(i)}] there exists a node $\{i\}$ (root node), $\omega_{0}$. Each sub-graph represents a single DM acting at time instances $t=0,\dots,T-1$, and there exists a finite number of sub-graphs $G_{p}(\hat{V},\hat{\mu})$ such that $\cup_{p=1}^{N}G_{p}\cup \{i\}=G$, where $G_{p}$s are isomorphic (see e.g., \cite{west1996introduction}) for all $p=1,\dots,N$, i.e., for every node with directed edges in each sub-graph there exists a unique node with identical directed edges in the corresponding sub-graphs, where $\hat{V}=\{0,\dots,T-1\}$, and $G_{p}^{k}$  refers to a node $k$ in $G_{p}$ for all $p=1,\dots,N$ and $k=0,\dots,T-1$, 
\item [{(ii)}] sharing of the information is symmetric across sub-graphs, i.e., for $p,s=1,\dots,N$, and $k,j=0,\dots,T-1$, and for every edge from a node $G_{p}^{k}$ to a node $G_{s}^{j}$, there exists an edge from a node $G_{p}^{k}$ to nodes $G_{-p}^{j}$ ,  where $G_{-p}^{j}$ denotes $(G_{1}^{j},\dots,G_{p-1}^{j},G_{p+1}^{j},\dots,G_{N}^{j})$, and also there exist edges from nodes $G_{-p}^{k}$ to a node $G_{p}^{j}$. 
\end{itemize}

\begin{figure}[b!]
\begin{centering}
\tikzstyle{place}=[circle,draw=blue!50,fill=blue!20,thick,
inner sep=0pt,minimum size=6mm]
\begin{tikzpicture}
\node at ( 0,1.5) [place] (2){${y}^{1}_{0}$};
\node at ( 0,0.5) [place] (3){$y^{2}_{0}$};
\node at ( 1,1.5) [place] (4){$y^{1}_{1}$};
\node at ( 1,0.5) [place] (5){$y^{2}_{1}$};
\node at (-1,1) [place] (1){$\omega_{0}$};
\draw [->] (1) to (2);
\draw [->] (1) to (3);
\draw [->] (2) to (4);
\draw [->] (3) to (5);
\node at (1.5,1.5) {...};
\node at (1.5,0.5) {...};
\end{tikzpicture}
\caption{A tree structure of a symmetric dynamic team.}
\end{centering}
\end{figure}
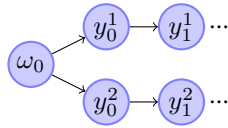
\begin{figure}
\begin{centering}
\tikzstyle{place}=[circle,draw=blue!50,fill=blue!20,thick,
inner sep=0pt,minimum size=6mm]
\begin{tikzpicture}
\node at ( 0,2.75) [place] (2){${y}^{1}_{0}$};
\node at ( 0,0.25) [place] (3){${y}^{2}_{0}$};
\node at ( 1,3.5) [place] (4){${y}^{1}_{1}$};
\node at ( 1,2) [place] (5){${y}^{1}_{2}$};
\node at (-1,1.5) [place] (1){$\omega_{0}$};
\node at ( 2,2.75) [place] (6){${y}^{1}_{3}$};
\node at ( 2,0.25) [place] (9){${y}^{2}_{3}$};
\node at ( 1,1) [place] (8){${y}^{2}_{1}$};
\node at ( 1,-0.5) [place] (7){${y}^{2}_{2}$};
\node at ( 3,0.25) [place] (11){${y}^{2}_{4}$};
\node at ( 3,2.75) [place] (10){${y}^{1}_{4}$};
\draw [->] (1) to (2);
\draw [->] (2) to (4);
\draw [->] (2) to (5);
\draw [->] (5) to (6);
\draw [->] (4) to (6);
\draw [->] (1) to (3);
\draw [->] (3) to (7);
\draw [->] (3) to (8);
\draw [->] (8) to (9);
\draw [->] (7) to (9);
\draw [->] (9) to (11);
\draw [->] (6) to (10);
\node at (3.5,2.75) {...};
\node at (3.5,0.25) {...};
\end{tikzpicture}
\caption{An example of the graph structure of a symmetric dynamic team.}
\end{centering}
\end{figure}
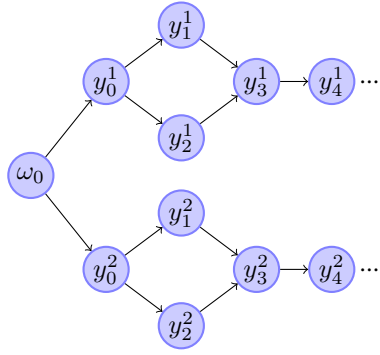

Now, we present an exchangeability hypothesis on the cost function. First, we recall the definition of an {\it exchangeable} finite set of random variables.
{\begin{definition}
Random variables $(x^{1},x^{2},\dots,x^{N})$ defined on a common probability space are \it{exchangeable} if for any permutation $\sigma$ of the set $\{1,\dots,N\}$ (a mapping $\sigma : \{1,\ldots,N\} \to \{1,\ldots,N\}$), 
\begin{flalign*}
&\scalemath{1}{{P}\bigg((x^{\sigma})^{1} \in A^{1},(x^{\sigma})^{2} \in A^{2},\dots,(x^{\sigma})^{N}\in A^{N}\bigg)={P}\bigg(x^{1} \in A^{1},x^{2} \in A^{2},\dots,x^{N}\in A^{N}\bigg)}
\end{flalign*}
for any measurable $\{A^{1},\dots, A^N\}$ and $(x^{\sigma})^{i}:=x^{\sigma(i)}$ for all $i\in \{1,\dots,N\}$.
\end{definition}}
{\begin{assumption}\label{assump:3.1}
For any permutation $\sigma$ of the set $\{1,\dots,N\}$, we have for all $\omega_{0}$
 \begin{flalign}\label{eq:3.3}
c((\underline{\zeta}^{\sigma})^{1:N},(\underline{u}^{\sigma})^{1:N})=c(\omega_{0}, \underline{\zeta}^{1:N},\underline{u}^{1:N}),
 \end{flalign}
 where $(\underline{\zeta}^{\sigma})^{1:N}=(\underline{\zeta}^{\sigma(1)},\dots,\underline{\zeta}^{\sigma(N)})$ and $(\underline{u}^{\sigma})^{1:N}=(\underline{u}^{\sigma(1)},\dots,\underline{u}^{\sigma(N)})$.
\end{assumption}}

Here, we recall some definitions and results from \cite[Section 3.3]{YukselSaldiSICON17} on convexity of static and dynamic team problems required to follow the result in this paper.
\begin{definition}\cite[Section 3.3]{YukselSaldiSICON17}
An $N$-DM team problem (static or dynamic) is convex in policies if for any two team policies ${\underline{\gamma}}_{T}^{1:N}$ and $\tilde{\underline{\gamma}}_{T}^{1:N}$ in the set $\{ \underline{\gamma}_{T}^{1:N} \in {\bf \Gamma}: J(\underline{\gamma}_{T}^{1:N})<\infty\},$ and for any $\alpha \in (0,1)$, we have
 \begin{equation*}
J_{T}(\alpha \underline{\gamma}_{T}^{1:N}+(1-\alpha)\tilde{\underline{\gamma}}_{T}^{1:N})\leq \alpha J_{T}(\underline{\gamma}_{T}^{1:N})+(1-\alpha)J_{T}(\tilde{\underline{\gamma}}_{T}^{1:N}).
\end{equation*}
\end{definition}

The above definition can also be applied to infinite-horizon and/or teams with countably infinite number of DMs. We recall sufficient conditions for convexity of static and dynamic team problems following \cite[Section 3.3]{YukselSaldiSICON17}.
\begin{theorem}\label{the:cp}\cite[Section 3.3]{YukselSaldiSICON17}
Consider a sequential team problems, and assume action spaces are convex, and $J(\underline{\gamma})<\infty$ for all $\underline{\gamma} \in {\bf \Gamma}$ (or alternatively, restrict the set to those leading to the finite cost). Then 
\begin{itemize}[wide]
\item[(i)] for static team problems convexity of the cost function in actions is sufficient for convexity of the team problem in policies,
\item [(ii)] for dynamic team problems with a static reduction, convexity of the team problem in policies is equivalent to the convexity of its static reduction. 
\item [(iii)] in particular, for partially nested dynamic teams with a static reduction (more generally, for stochastically partially nested team problems \cite[Section 3.3]{YukselSaldiSICON17}) if the cost function is convex in actions then for the reduced team problem with an equivalent information structure is convex on ${\bf \Gamma}$.
\end{itemize}
\end{theorem}

The conditions above, however, are only sufficient conditions \cite[Example 1]{YukselSaldiSICON17}. We note however that as a Corollary for (ii) above, for the LQG setup, under partial nestedness, convexity in policies hold as a consequence of Radner’s theorem; we will study this case in Section \ref{sec:LQG}. On the other hand, not all LQG problems are convex: the celebrated counterexample of Witsenhausen \cite{wit68} demonstrates that under non-classical information structures, even LQG problems may not be convex and optimal policies may not be linear.
\subsection{Optimality of symmetric policies for convex dynamic teams with a symmetric information structure}
In the following, we define notions of exchangeable and symmetrically optimal teams analogous to \cite{sanjari2018optimal, sinacdc2018optimal} for dynamic teams. 
\begin{definition}(Exchangeable teams)\label{def:sym}\\
An $N$-DM team is \textit{exchangeable} if the value of the expected cost function is invariant under every permutation of policies of DMs, i.e., $J_{T}(\underline{\gamma}_{T}^{1},\underline{\gamma}_{T}^{2},\dots,\underline{\gamma}_{T}^{N})=J_{T}((\underline{\gamma}_{T}^{\sigma})^{1},\dots, (\underline{\gamma}_{T}^{\sigma})^{N})$.
\end{definition}

\begin{definition}(Symmetrically optimal teams)\label{Def:ost}\\
A team is \textit{symmetrically optimal}, if for every given policy $\underline{\gamma}_{T}=(\underline{\gamma}_{T}^{1},\dots,\underline{\gamma}_{T}^{N})$, there exists an identically symmetric policy (i.e., each DM has the same policy, $\underline{\tilde{\gamma}}_{T}=(\underline{\tilde{\gamma}}_{T}^{1},\dots,\underline{\tilde{\gamma}}_{T}^{N})$ , and $\underline{\tilde{\gamma}}_{T}^{i}=\underline{\tilde{\gamma}}_{T}^{j}$ for all $i,j=1,\dots,N$) which performs at least as good as the given policy.
\end{definition}
\begin{remark}\label{rem:exchang}
The concepts of exchangeable and symmetrically optimal dynamic teams in this paper are generalizations of those for static teams in \cite{sanjari2018optimal, sinacdc2018optimal}. However, here, the value of the expected cost function may not be invariant under exchanging $\gamma^{i}_{t}$ with $\gamma^{j}_{k}$ for $k\not =t$, $k,t=0,\dots,T-1$, and for $i,j=1,\dots,N$.
\end{remark}

Here, we give a characterization for exchangeable and symmetrically optimal dynamic teams. 
\begin{theorem}\label{the:3.1}
Consider dynamic team problems with a symmetric information structure under Assumption \ref{assump:3.1}. If
\begin{itemize}
\item [{(a)}] action spaces $\mathbb{U}_{t}$ are convex for all $t=0, \dots, T-1$, 
\item [{(b)}] $(\underline\zeta^{1}, \dots,\underline\zeta^{N})$ are exchangeable, 
\item  [{(c)}] for all policies $\gamma \in \bf{\Gamma}$, and for all $A=A^{1}\times \dots \times A^{N}$ where $A^{i} \in \mathcal{Y}^{i}$, 
\end{itemize}
\begin{flalign}
\scalemath{1}{\prod_{t=0}^{T-1}{P}\bigg({y}^{1:N}_{t}}&\scalemath{1}{\in A\bigg|x_{0}^{1:N},{\zeta}^{1:N}_{0:t-1},{y}_{0:t-1}^{1:N},{\gamma}^{1}_{0}({y}^{1}_{0}),\dots, {\gamma}^{1}_{t-1}({y}^{1}_{t-1}), {\gamma}^{1}_{0}({y}^{2}_{0}), \dots, {\gamma}^{N}_{t-1}({y}^{N}_{t-1})\bigg)}\nonumber\\
&\scalemath{1}{=\prod_{t=0}^{T-1}\prod_{i=1}^{N}{P}\bigg({y}^{i}_{t}\in A^{i}\bigg|x_{0}^{i},{\zeta}^{i}_{0:t-1},{y}_{\downarrow t}^{\downarrow\downarrow i},{\gamma}^{\downarrow\downarrow i}_{\downarrow t}({y}^{\downarrow\downarrow i}_{\downarrow t})\bigg)}\label{eq:a},
\end{flalign}
where $y_{\downarrow t}^{\downarrow\downarrow i}:=\{y_{p}^{j}|u^{j}_{p}~\text{affects}~y^{i}_{t}~\forall~p=0,\dots,t-1~\text{and}~\forall j=1,\dots, N\}$ and $({\gamma}^{\downarrow\downarrow i}_{\downarrow t}({y}^{\downarrow\downarrow i}_{\downarrow t}))$ can be defined similarly, 
\begin{itemize}
\item [(i)] then, the team problem is exchangeable.
\item [(ii)]  Furthermore, if the team problem is convex in policies (see Theorem \ref{the:cp}), then the team is symmetrically optimal.
\end{itemize}
\end{theorem}
\begin{proof}
We first show that for any permutation $\sigma \in S$, $J_{T}((\underline{\gamma}_{T}^{\sigma})^{1},\dots, (\underline{\gamma}_{T}^{\sigma})^{N})=J_{T}(\underline{\gamma}_{T}^{1},\dots, \underline{\gamma}_{T}^{N})$, i.e., the team is exchangeable. We have,
\begin{flalign}
&J_{T}\left((\underline{\gamma}_{T}^{\sigma})^{1},\dots, (\underline{\gamma}_{T}^{\sigma})^{N}\right)\nonumber\\
&=\int {c}\left(\underline{\zeta}^{1:N}, (\underline{\gamma}_{T}^{\sigma})^{1}(\underline{y}^{1}),\dots,(\underline{\gamma}_{T}^{\sigma})^{N}(\underline{y}^{N})\right){\mu}(dx_{0}^{1:N}, d{\zeta}^{1:N}_{0:T-1})\label{eq:3.6}\\
&~~~~~~\:\:\:\:\:\:\times\prod_{t=0}^{T-1}\prod_{i=1}^{N}{P}\left(d{y}^{i}_{t}\middle|x_{0}^{i}, \zeta_{0:t-1}^{i},{y}_{\downarrow t}^{\downarrow\downarrow i},({\gamma}^{\sigma})^{\downarrow\downarrow i}_{\downarrow t}({y}^{\downarrow\downarrow i}_{\downarrow t})\right)\nonumber\\
&=\int {c}\left((\underline{\zeta}^{\sigma})^{1:N}, (\underline{\gamma}_{T}^{\sigma})^{1}((\underline{y}^{\sigma})^{1}),\dots,(\underline{\gamma}_{T}^{\sigma})^{N}((\underline{y}^{\sigma})^{N})\right){\mu}(d(x_{0}^{\sigma})^{1:N}, d({\zeta}^{\sigma}_{0:T-1})^{1:N})\label{eq:3.7}\\
&~~~~~~\:\:\:\:\:\:\times\prod_{t=0}^{T-1}\prod_{i=1}^{N}{P}\left(d({y}^{\sigma})^{i}_{t}\middle|(x_{0}^{\sigma})^{i}, (\zeta_{0:t-1}^{\sigma})^{i},{(y^{\sigma})}_{\downarrow t}^{\downarrow\downarrow i},({\gamma}^{\sigma})^{\downarrow\downarrow i}_{\downarrow t}(({y^{\sigma}})^{\downarrow\downarrow i}_{\downarrow t})\right)\nonumber\\
&=\int {c}\left(\underline{\zeta}^{1:N}, \underline{\gamma}^{1}_{T}(\underline{y}^{1}),\dots,\underline{\gamma}^{N}_{T}(\underline{y}^{N})\right){\mu}(dx_{0}^{1:N}, d{\zeta}^{1:N}_{0:T-1})\label{eq:3.8}\\
&~~~~~~\:\:\:\:\:\:\times\prod_{t=0}^{T-1}\prod_{i=1}^{N}{P}\left(d{y}^{i}_{t}\middle|x_{0}^{i}, \zeta_{0:t-1}^{i},{y}_{\downarrow t}^{\downarrow\downarrow i},{\gamma}^{\downarrow\downarrow i}_{\downarrow t}({y}^{\downarrow\downarrow i}_{\downarrow t})\right)\nonumber\\
&=J_{T}(\underline{\gamma}_{T}^{1},\dots, \underline{\gamma}_{T}^{2})\nonumber,
\end{flalign}
where \eqref{eq:3.6} follows from condition (c). Equality \eqref{eq:3.7} follows from  exchanging $\underline{y}^{i}$, $\underline{\zeta}^{i}$ with $(\underline{y}^{\sigma})^{i}$, $(\underline{\zeta}^{\sigma})^{i}$ by relabeling them, respectively. Since the information structure is symmetric, \eqref{eq:3.3} and condition (b) imply \eqref{eq:3.8}. Hence, the team is exchangeable.  Let $\underline{\gamma}^{*}_{T}=(\underline{\gamma}_{T}^{1*},\dots, \underline{\gamma}_{T}^{N*})$ be a given policy. Consider $\underline{\tilde{\gamma}}_{T}$ as a convex combination of all possible permutations of policies by averaging them. Since action spaces are convex by condition (a), $\underline{\tilde{\gamma}}_{T}$ is a control policy. Following from convexity of the cost function in policies, we have 
 \begin{flalign*}
J_{T}(\underline{\tilde{\gamma}}_{T})&:=J_{T}(\sum_{\sigma\in S}\frac{1}{|S|}\underline{\gamma}^{*,\sigma}_{T}) \leq \sum_{\sigma\in S}\frac{1}{|S|}J_{T}(\underline{\gamma}^{*,\sigma}_{T}) 
=J_{T}(\underline{\gamma}^{*}_{T}),
 \end{flalign*}
where $|S|$ denotes the cardinality of the set $S$ and the inequality above follows from the hypothesis that the team problem is convex on ${\bf \Gamma}$ and the last equality follows from exchangeability of the team problem. This implies that the team is symmetrically optimal and completes the proof.
\end{proof}

Examples will be given in Section \ref{sec:mft} and Section \ref{sec:3a} where Theorem \ref{the:3.1} can be applied.
Here, we present the result for a class of problems that admit a static reduction (see \cite[Section 3.7]{YukselBasarBook}, \cite [Section 1.2]{YukselSaldiSICON17}, \cite{ho1973equivalence, wit88}). 

\begin{lemma}\label{lem:stareduction}
Consider a dynamic team problem with a symmetric partially nested information structure (see Definition \ref{def:s}) which admits a static reduction. Under Assumption \ref{assump:3.1}, and Assumptions (a), (b), (c) of Theorem \ref{the:3.1}, if the cost function is jointly convex in $\underline{u}^{ 1},\dots,\underline{u}^{N}$ ${P}$-almost surely, then the team is symmetrically optimal.
\end{lemma}
We note again that here by symmetry, we mean symmetry across the decision makers.
\begin{proof}
The proof follows from Theorem \ref{the:cp}(iii) and Theorem \ref{the:3.1} since the team is convex on ${\bf \Gamma}$ under the static reduction which is equivalent to the dynamic problem.
\end{proof}

Hence, it follows that if a static reduction of an exchangeable, symmetrically optimal, dynamic team exists, then it is exchangeable and symmetrically optimal.

\section{Convex mean-field teams with a symmetric information structure}\label{sec:mft}

In the following, we establish global optimality results for convex mean-field teams with a symmetric information structure (that is not necessarily partially nested).  

Define state dynamics and observations as 
\begin{flalign}
&x_{t+1}^{i}=f_{t}(x_{t}^{i},u_{t}^{i},w_{t}^{i}),\label{eq:mfdynamics}\\
&y_{t}^{i}=h_{t}(x_{0:t}^{i},u_{0:t-1}^{i},v_{0:t}^{i})\label{eq:mfobs},
\end{flalign}
where functions $f_{t}$ and $h_{t}$ are measurable functions. The information structure of DM$^{i}$ at time $t$ is $I_{t}^{i}=\{{y}^{i}_{t}\}$, and $\zeta_{t}^{i}:=(w_{t}^{i}, v_{t}^{i})$ (with $\zeta_{0}^{i}:=(x_{0}^{i}, w_{0}^{i}, v_{0}^{i})$) denotes uncertainty corresponding to dynamics and observations at time $t$ for DM$^i$ which are exogenous random vectors in a standard Borel space. Denote $\mathbb{X}\subseteq \mathbb{R}^{m}$, $\mathbb{U}\subseteq \mathbb{R}^{m^{\prime}}$, $\mathbb{Y}\subseteq \mathbb{R}^{m^{\prime\prime}}$, $\mathbb{W}$, and $\mathbb{V}$ as the state space, action space, observation space and the space of disturbances of dynamics and observations of DMs at each time instances $t=0, \dots, T-1$, respectively, where $m$, $m^\prime$, and $m^{\prime\prime}$ are positive integers.
 \begin{itemize}[wide = 0pt]
  \item[ \textbf{Problem} \bf{($\mathcal{P}_{T}^{N, \text{MF}}$)}:] Consider $N$-DM teams with the expected cost function of $\underline{\gamma}_{T}^{1:N}$ as 
  \begin{flalign}
&\scalemath{1}{J_{T}^{N}(\underline{\gamma}_{T}^{N})=\frac{1}{N}\sum_{t=0}^{T-1}\sum_{i=1}^{N} E^{\underline{\gamma}_{T}^{1:N}}\left[c\left(\omega_{0}, x_{t}^{i},u_{t}^{i},\frac{1}{N}\sum_{p=1}^{N}u_{t}^{p},\frac{1}{N}\sum_{p=1}^{N}x_{t}^{p}\right)\right]}\label{eq:mfcost},
\end{flalign}
where $\omega_{0}:(\Omega,\mathcal{F}) \to (\Omega_{0},\mathcal{F}_{0})$ is an exogenous random vector in the standard Borel space and $\underline{\gamma}_{T}^{1:N}=\gamma^{1:N}_{0:T-1}$, and the cost function satisfies the following assumption.

  \item[\textbf{Problem} \bf{($\mathcal{P}_{T}^{\infty, \text{MF}}$)}:]
Consider mean-field teams with the expected cost function of $\underline{\gamma}_{T}$ as
\begin{flalign}
&\scalemath{1}{J_{T}^{\infty}(\underline{\gamma}_{T})=\limsup\limits_{N\rightarrow \infty}J_{T}^{N}(\underline{\gamma}_{T})}\label{eq:mfcost1},
\end{flalign}
where $J_{T}^{N}(\cdot)$ is defined in \eqref{eq:mfcost}, $\underline{\gamma}_{T}^{i}=\gamma^{i}_{0:T-1}$ for $i \in \mathbb{N}$  and $\underline{\gamma}_{T}=\{\underline{\gamma}_{T}^{i}\}_{i \in \mathbb{N}}$.
\end{itemize}
\begin{assumption}\label{assump:c} Assume
 \begin{itemize}
 \item[(a)] function $f_{t}:\mathbb{X} \times \mathbb{U} \times \mathbb{W} \to \mathbb{X}$ is continuous in its first and second arguments for all $w_{t}^{i}$ and for each $i\in \mathbb{N}$ and uniformly bounded, 
 \item[(b)] function $h_{t}:\prod_{k=0}^{t}\mathbb{X} \times \prod_{k=0}^{t-1}\mathbb{U} \times \prod_{k=0}^{t}\mathbb{V} \to \mathbb{Y}$ is continuous in states and actions for all $v_{0:t}^{i}$ and for each $i\in \mathbb{N}$, and
 \item [(c)] the cost function in \eqref{eq:mfcost}, $c:\Omega_{0}\times \mathbb{X}\times \mathbb{U} \times \mathbb{U} \times \mathbb{X} \to \mathbb{R}_{+}$, is continuous in its second, third, fourth, and fifth arguments for all $\omega_{0}$.
 \end{itemize}
 \end{assumption}
\subsection{Mean-field optimal policies as limits of optimal $N$-DM teams}
In the following, we first establish global optimality results under Assumption \ref{assump:2} (see Theorem \ref{the:mft}), then we establish the result under a more relaxed assumption, Assumption \ref{assump:ergodic} (see Theorem \ref{the:mftergodic}):
\begin{assumption}\label{assump:2}
Assume 
\begin{itemize}
\item[(i)] $(x_{0}^{1},x_{0}^{2}, \dots)$ are i.i.d. random vectors conditioned on $\omega_{0}$,  
\item[(ii)] for $t=0,\dots,T-1$, $\{w^{i}_{t}\}_{i\in \mathbb{N}}$ are i.i.d. random vectors, and for $i\in \mathbb{N}$, $\{w^{i}_{t}\}_{t=0}^{T-1}$ are mutually independent, and independent of $\omega_{0}$ and $(x_{0}^{1},x_{0}^{2}, \dots)$. For $t=0,\dots,T-1$, $\{v^{i}_{t}\}_{i\in \mathbb{N}}$ are i.i.d. random vectors, and for $i\in \mathbb{N}$, $\{v^{i}_{t}\}_{t=0}^{T-1}$ are mutually independent, and independent of $\omega_{0}$, $(x_{0}^{1},x_{0}^{2}, \dots)$, and $w^{i}_{t}$s for $i\in \mathbb{N}$ and $t=0,\dots,T-1$.
\end{itemize}
\end{assumption}
\begin{assumption}\label{assump:ergodic}Assume that conditioned on $\omega_{0}$, $(x_{0}^{1},x_{0}^{2}, \dots)$ are exchangeable random vectors.
\end{assumption}

Later on we will establish an existence theorem under Assumption \ref{assump:2}, and we note that the proof under Assumption \ref{assump:2} is more direct. This is why two separate theorems will be presented, and the proof of the latter will be built on that of the former.

\begin{lemma}\label{lem:mflem}
Consider a team defined as ($\mathcal{P}_{T}^{N,\text{MF}}$) with a symmetric information structure. Assume the team problem is convex in policies. Let the action space be compact and convex for each decision makers. Under Assumption \ref{assump:c} and Assumption \ref{assump:2}, the team is symmetrically optimal.
\end{lemma}
\begin{proof}
The proof follows from Theorem \ref{the:3.1}.
\end{proof}

 \begin{theorem}\label{the:mft}
 Consider a team defined as  ($\mathcal{P}_{T}^{\infty,\text{MF}}$) with ($\mathcal{P}_{T}^{N,\text{MF}}$) having a symmetric information structure for every $N$. Assume for every $N$ the team problem is convex in policies. Let the action space be compact and convex for each DM. Under Assumption \ref{assump:c}, and Assumption \ref{assump:2}, if there exists a sequence of optimal policies for ($\mathcal{P}_{T}^{N,MF}$), $\{\underline{\gamma}^{*,N}_{T}\}_N$, which converges (for every DM due to the symmetry) pointwise to $\underline{\gamma}^{*,\infty}_{T}$ as $N \to \infty$, then $\underline{\gamma}^{*,\infty}_{T}$  (which is identically symmetric) is an optimal policy for ($\mathcal{P}_{T}^{\infty,MF}$).
\end{theorem}
\begin{proof}
 Following from Lemma  \ref{lem:mflem}, one can consider a sequence of $N$-DM teams which are symmetrically optimal that defines ($\mathcal{P}_{T}^{N,\text{MF}}$) and whose limit is identified with $(\mathcal{P}_{T}^{\infty,\text{MF}})$. Define 
 \begin{flalign}
\scalemath{1}{ Q_{N}(B):=\frac{1}{N}\sum_{i=1}^{N}\delta_{\beta_{N}^{i}}(B)\:\:\:\:\:\:\text{where}\:\:\:\:\:\:\beta_{N}^{i}:=(\underline{\gamma}^{*,N}_{T}(\underline{y}^{i}),\underline{y}^{i},\underline{\zeta}^{i})},\label{eq:q}\\
\scalemath{1}{ \tilde{Q}_{N}(B):=\frac{1}{N}\sum_{i=1}^{N}\delta_{{\beta}_{\infty}^{i}}(B)\:\:\:\:\:\ \text{where}\:\:\:\:\:\:{\beta}_{\infty}^{i}:=(\underline{\gamma}^{*,\infty}_{T}(\underline{y}^{i}),\underline{y}^{i},\underline{\zeta}^{i})}\nonumber,
 \end{flalign}
 where $\delta_{Y}(\cdot)$ denotes the Dirac measure for any random vector $Y$, and $B \in \mathcal{Z}:={\bf{U}}\times {\bf{Y}} \times {\bf{S}}$, ${\bf{U}}:=(\prod_{t=0}^{T-1}\mathbb{U})$, ${\bf{Y}}:=(\prod_{t=0}^{T-1}\mathbb{Y})$, ${\bf{S}}:=(\prod_{t=0}^{T-1}{\mathbb{S}})=\mathbb{X}\times (\prod_{t=0}^{T-1}{\mathbb{W}\times \mathbb{V}})$, ${\bf X}=(\prod_{t=0}^{T-1}\mathbb{X})$, $\underline{y}^{i}=(y_{0}^{i},\dots,y_{T-1}^{i})$, and $\underline{\zeta}^{i}:=(\zeta^{i}_{0},\dots,\zeta_{T-1}^{i})$.
 
 In the following, first, we show that conditioned on $\omega_{0}$, $Q_{N}$ converges ${P}$-almost surely to $Q=\mathcal{L}(\beta^1_{\infty}|\omega_{0})$ in $w$-$s$ topology (coarsest topology on $\mathcal{P}({\bf{U}}\times {\bf{Y}} \times {\bf{S}})$  under which $\int f(u,y,\zeta)Q_{N}(du, dy, d\zeta): \mathcal{P}({\bf{U}}\times {\bf{Y}} \times {\bf{S}}) \to \mathbb{R}$ is continuous for every measurable and bounded function $f$ which is continuous in $u$ and $y$ but need not to be continuous in $\zeta$ (see e.g., \cite{Schal} and \cite[Theorem 5.6]{yuksel2018general})). Then, we show that 
 \begin{equation*}
 \limsup\limits_{N\rightarrow \infty}J_{T}^{N}(\underline{\tilde{\gamma}}_{T}^{*,N})=J_{T}^{\infty}(\underline{\tilde{\gamma}}^{*,\infty}_{T}),
 \end{equation*}
 where $\underline{\tilde{\gamma}}_{T}^{*,N}:=(\underline{{\gamma}}_{T}^{*,N},\underline{{\gamma}}_{T}^{*,N},\dots,\underline{{\gamma}}_{T}^{*,N})$ and $\underline{\tilde{\gamma}}^{*,\infty}_{T}:=(\underline{{\gamma}}^{*,\infty}_{T},\underline{{\gamma}}^{*,\infty}_{T},\dots)$.
 \begin{itemize}[wide]
 \item [\bf{(Step 1):}]
 In this step, we show that conditioned on $\omega_{0}$, $Q_{N}$ converges ${P}$-almost surely to $Q$ in $w$-$s$ topology. First, we show that for every continuous and bounded function $g$ in actions and observations, for every $\omega_{0}$ on a set of ${P}$-measure one,
\begin{flalign}
&\scalemath{1}{{{P}\left(\bigg\{\omega \in \Omega\bigg|\lim\limits_{N\rightarrow \infty}\left(\frac{1}{N}\sum_{i=1}^{N}\left[g(\underline{\gamma}^{*,N}_{T}(\underline{y}^{i}),\underline{y}^{i},\underline{\zeta}^{i})-g(\underline{\gamma}^{*,\infty}_{T}(\underline{y}^{i}),\underline{y}^{i},\underline{\zeta}^{i})\right]\right)=0\bigg\}\middle|\omega_{0}\right)=1}}\label{eq:4.1.1}.
\end{flalign}
Following from  symmetry of the information structure and Lemma \ref{lem:mflem}, every DM applies an identical optimal policy $\underline{\gamma}_{T}^{*,N}$ and since functions $f_{t}$ and $h_{t}$ are identical for each DM, conditioned on $\omega_{0}$, $(\underline{\gamma}^{*,N}_{T}(\underline{y}^{i}),\underline{y}^{i},\underline{\zeta}^{i})$ and $(\underline{\gamma}^{*,\infty}_{T}(\underline{y}^{i}),\underline{y}^{i},\underline{\zeta}^{i})$ are i.i.d. random vectors. For every $\epsilon>0$ and for every function $g$ continuous and bounded in actions and observations, we have $P$-almost surely
\begin{flalign}
&\scalemath{1}{\lim\limits_{N\rightarrow \infty}{P}\left(\left|\int g dQ_{N}-\int g d\tilde{Q}_{N}\right|\geq\epsilon\bigg|\omega_{0}\right)}\nonumber\\ 
&\scalemath{1}{\leq\epsilon^{-1}\lim\limits_{N\rightarrow \infty}\frac{1}{N}\sum_{i=1}^{N} E\left[\left|g(\underline{\gamma}^{*,N}_{T}(\underline{y}^{i}),\underline{y}^{i},\underline{\zeta}^{i})-g(\underline{\gamma}^{*,\infty}_{T}(\underline{y}^{i}),\underline{y}^{i},\underline{\zeta}^{i})\right|\bigg| \omega_{0}\right]}\label{eq:y4.1.1}\\
&\scalemath{1}{=\epsilon^{-1}\lim\limits_{N\rightarrow \infty}  E\left[\left|g(\underline{\gamma}^{*,N}_{T}(\underline{y}^{i}),\underline{y}^{i},\underline{\zeta}^{i})-g(\underline{\gamma}^{*,\infty}_{T}(\underline{y}^{i}),\underline{y}^{i},\underline{\zeta}^{i})\right|\bigg| \omega_{0}\right]}\label{eq:y199.5.1}\\
&\scalemath{1}{=\epsilon^{-1} E\left[\lim\limits_{N\rightarrow \infty}\left|g(\underline{\gamma}^{*,N}_{T}(\underline{y}^{i}),\underline{y}^{i},\underline{\zeta}^{i})-g(\underline{\gamma}^{*,\infty}_{T}(\underline{y}^{i}),\underline{y}^{i},\underline{\zeta}^{i})\right|\bigg| \omega_{0}\right]=0} \label{eq:y199.5.2},
\end{flalign}
where \eqref{eq:y4.1.1} follows from Markov's inequality, the triangle inequality and the definition of the empirical measure, and \eqref{eq:y199.5.1} follows from the fact that $(\underline{\gamma}^{*,N}_{T}(\underline{y}^{i}),\underline{y}^{i},\underline{\zeta}^{i})$ and $(\underline{\gamma}^{*,\infty}_{T}(\underline{y}^{i}),\underline{y}^{i},\underline{\zeta}^{i})$ are i.i.d. random vectors. Since  $g$ is bounded and continuous, the dominated convergence theorem implies \eqref{eq:y199.5.2}. Hence, for every subsequence, there exists a subsubsequence such that $P$-almost surely  ${{P}\left(\{\omega \in \Omega|\lim\limits_{N \to \infty}\left(\int g dQ_{N}-\int g d\tilde{Q}_{N}\right)=0\}\big|\omega_{0}\right)}=1$.

 Now, we show that conditioned on $\omega_{0}$, $\{\tilde{Q}_{N}\}_{N}$ converges weakly to $Q$ ${P}$-almost surely. Since conditioned on $\omega_{0}$,  $(\underline{\gamma}^{*,\infty}_{T}(\underline{y}^{i}),\underline{y}^{i},\underline{\zeta}^{i})$ are i.i.d. random vectors, the strong law of large numbers implies that $P$-almost surely
\begin{flalign}
\scalemath{0.93}{{P}\left(\bigg\{\omega\in \Omega\bigg|\lim\limits_{N\rightarrow \infty}\left(\frac{1}{N}\sum_{i=1}^{N}g\bigg(\underline{\gamma}^{*,\infty}_{T}(\underline{y}^{i}),\underline{y}^{i},\underline{\zeta}^{i}\bigg)- E\bigg[g(\underline{\gamma}^{*,\infty}_{T}(\underline{y}^{1}),\underline{y}^{1},\underline{\zeta}^{1})\bigg|\omega_{0}\bigg]\right)=0\bigg\}\middle| \omega_{0}\right)=1},\label{eq:stlaw}
\end{flalign}
hence, 
${{P}\left(\{\omega \in \Omega|\lim\limits_{N\rightarrow \infty}\left(\int g d\tilde{Q}_{N}-\int g dQ\right)=0\}\big|\omega_{0}\right)}=1$ $P$-almost surely.

Hence, through choosing a suitable subsequence, for every $\omega_{0}\in \Omega_{0}$ on a set of ${P}$-measure one, for every function $g$ continous and bounded in actions and observations and measurable and bounded in uncertainty and initial states
 \begin{flalign}
\scalemath{1}{\lim\limits_{N\rightarrow \infty}\left|\int g dQ_{N}-\int g dQ\right|}&\scalemath{1}{\leq \lim\limits_{N\rightarrow \infty} \bigg(\left|\int g dQ_{N}-\int g d\tilde{Q}_{N}\right|+\left|\int g d\tilde{Q}_{N}-\int g dQ\right| \bigg)}=0\nonumber,
\end{flalign}
hence, conditioned on $\omega_{0}$, $Q_{N}$ converges weakly to $Q$ ${P}$-almost surely. We note that the convergence is weakly, but since $\underline{\zeta}^{i}$s are exogenous with a fixed probability measure, the convergence is also in the $w$-$s$ topology.
\item [\bf{(Step 2):}] 
Following from \eqref{eq:mfdynamics} and \eqref{eq:mfobs}, we have
\begin{flalign}
x_{t}^{i}&=f_{t-1}(f_{t-2}(\dots f_{0}(x_{0}^{i}, u_{0}^{i},w_{0}^{i})),u_{t-1}^{i},w_{t-1}^{i})=\tilde{f}_{t-1}(\underline{\zeta}^{i},u_{0:t-1}^{i}),\label{eq:stateed}\\
y_{t}^{i}&=h_{t}(x_{0}^{i},\tilde{f}_{0}(\underline{\zeta}^{i},u_{0}^{i}), \dots, \tilde{f}_{t-1}(\underline{\zeta}^{i},u_{0:t-1}^{i}), u_{0:t-1}^{i}, v_{0:t}^{i})=\tilde{h}_{t}(\underline{\zeta}^{i},u_{0:t-1}^{i})\label{eq:obsred},
\end{flalign}
where following from Assumption \ref{assump:c}, $\tilde{f}_{t-1}$ and $\tilde{h}_{t}$ are continuous in actions. Hence, under Assumption \ref{assump:c}(c), we have
\begin{flalign}
&\scalemath{1}{\frac{1}{N}\sum_{i=1}^{N}\sum_{t=0}^{T-1} E^{\underline{\gamma}_{T}^{*,1:N}}\left[c\left(\omega_{0}, x_{t}^{i},u_{t}^{i},\frac{1}{N}\sum_{p=1}^{N}u_{t}^{p},\frac{1}{N}\sum_{p=1}^{N}x_{t}^{p}\right)\right]}\nonumber\\
&\scalemath{1}{=\frac{1}{N}\sum_{i=1}^{N} E\left[\tilde{c}\left(\omega_{0},\underline{\zeta}^{i}, \underline{\gamma}^{*,N}_{T}(\underline{y}^{i}),\frac{1}{N}\sum_{i=1}^{N}\underline{\gamma}^{*,N}_{T}(\underline{y}^{i}),\frac{1}{N}\sum_{i=1}^{N}\Lambda(\underline{\gamma}^{*,N}_{T}(\underline{y}^{i}),\underline{\zeta}^{i})\right)\right]}\label{eq:star1},
\end{flalign}
where \eqref{eq:star1} is true following from \eqref{eq:stateed} for some functions $\tilde{c}:\Omega_{0}\times {\bf S} \times {\bf U} \times {\bf U} \times {\bf X} \to \mathbb{R}_{+}$ which are continuous in its last three arguments and a function $\Lambda: {\bf U} \times {\bf S} \to {\bf X}$ which is continuous in actions. Hence, by induction and rewriting observations as a functions of policies of the past DMs ($\gamma_{\downarrow t}^{*, N}(y_{\downarrow t}^{i})$) since $\underline{\gamma}_{T}^{*,N}$ converges to $\underline{\gamma}_{T}^{*,\infty}$, the induced cost by $\underline{\gamma}_{T}^{*,N}$ also converges to the cost induced by $\underline{\gamma}_{T}^{*,\infty}$ ${P}$-almost surely.
\item [\bf{(Step 3):}] 
We have
\begin{flalign}
&\scalemath{1}{\limsup\limits_{N\rightarrow \infty}\frac{1}{N}\sum_{i=1}^{N}\sum_{t=0}^{T-1} E^{\underline{\gamma}_{T}^{*,1:N}}\left[c\left(\omega_{0}, x_{t}^{i},u_{t}^{i},\frac{1}{N}\sum_{p=1}^{N}u_{t}^{p},\frac{1}{N}\sum_{p=1}^{N}x_{t}^{p}\right)\right]}\nonumber\\
&\scalemath{1}{=\limsup\limits_{N\rightarrow \infty}}\scalemath{1}{\frac{1}{N}\sum_{i=1}^{N} E\left[\tilde{c}\left(\omega_{0},\underline{\zeta}^{i}, \underline{\gamma}^{*,N}_{T}(\underline{y}^{i}),\frac{1}{N}\sum_{i=1}^{N}\underline{\gamma}^{*,N}_{T}(\underline{y}^{i}),\frac{1}{N}\sum_{i=1}^{N}\Lambda(\underline{\gamma}^{*,N}_{T}(\underline{y}^{i}),\underline{\zeta}^{i})\right)\right]}\label{eq:star}\\
&\scalemath{1}{\geq \liminf\limits_{N\rightarrow \infty} E\bigg[ E\bigg[}\scalemath{1}{\int_{\cal{Z}}\tilde{c}\left(\omega_{0},\zeta, u,\int_{\bf U}uQ_{N}(du \times {\bf{Y}}\times {\bf{S}}),\int_{{\bf{U}} \times {\bf{S}}}\Lambda Q_{N}(du \times {\bf{Y}}\times d\zeta)\right)}\label{eq:31y}\\
&~~~~~\:\:\:\:\:\:\:\:\:\:\:\:\:\:\:\scalemath{1}{\times Q_{N}(du,dy,d\zeta)\bigg|\omega_{0}\bigg]\bigg]}\nonumber\\
&\scalemath{1}{\geq  E\bigg[ E\bigg[\liminf\limits_{N\rightarrow \infty}}{\int_{\cal{Z}}\tilde{c}\left(\omega_{0}, \zeta, u,\int_{{\bf{U}}}uQ_{N}(du \times {\bf{Y}}\times {\bf{S}}),\int_{{\bf{U}} \times {\bf{S}}}\Lambda Q_{N}(du \times {\bf{Y}}\times d\zeta)\right)}\label{eq:32y}\\
&~~~~~\:\:\:\:\:\:\:\:\:\:\:\:\:\:\:\scalemath{1}{\times  Q_{N}(du,dy,d\zeta)\bigg|\omega_{0}\bigg]\bigg]}\nonumber\\
&\scalemath{0.97}{\geq E\bigg[ E\bigg[\int_{\cal{Z}}\tilde{c}\left(\omega_{0},\zeta, u,\int_{{\bf{U}}}uQ(du \times {\bf{Y}}\times {\bf{S}}),\int_{{\bf{U}} \times {\bf{S}}}\Lambda Q(du \times {\bf{Y}}\times d\zeta)\right)}\scalemath{1}{Q(du,dy,d\zeta)\bigg|\omega_{0}\bigg]\bigg]}\label{eq:2121},
\end{flalign}
where \eqref{eq:star} follows from \eqref{eq:star1}. Inequality \eqref{eq:31y} follows from \eqref{eq:q} and replacing limsup with liminf, and \eqref{eq:32y} follows from Fatou's lemma. In the following, we justify \eqref{eq:2121}.  Since conditioned on $\omega_{0}$, $Q_{N}$ converges  weakly to $Q$ ${P}$-almost surely, we have $Q_{N}(du \times {\bf{Y}}\times {\bf{S}})$ converges weakly to $Q(du\times {\bf{Y}}\times {\bf{S}})$ ${P}$-almost surely conditioned on $\omega_{0}$, hence, the compactness of ${\bf{U}}$ implies that conditioned on $\omega_{0}$, ${P}$-almost surely
\begin{equation}\label{eq:ucon}
\scalemath{1}{\frac{1}{N}\sum_{i=1}^{N} \underline{\gamma}^{*,N}_{T}(\underline{y}^{i})=\int_{{\bf{U}}}uQ_{N}(du \times {\bf{Y}}\times {\bf{S}}) \xrightarrow{{N \to \infty}} \int_{{\bf{U}}}uQ(du \times {\bf{Y}}\times {\bf{S}})= E[\underline{\gamma}^{*,\infty}_{T}(\underline{y}^{1})|\omega_{0}]}.
\end{equation}
 Since conditioned on $\omega_{0}$, $Q_{N}$ converges  weakly to $Q$ ${P}$-almost surely, we have $Q_{N}(du \times {\bf{Y}}\times d\zeta)$ converges ${P}$-almost surely to $Q(du\times {\bf{Y}}\times d\zeta)$ in $w$-$s$ topology conditioned on $\omega_{0}$. Following from \eqref{eq:stateed}, since $f_{t}$s are bounded and continuous in actions, $\Lambda$ is bounded and continuous in actions, hence, this implies that conditioned on $\omega_{0}$, ${P}$-almost surely
\begin{equation}
\scalemath{1}{\int_{{\bf{U}} \times {\bf{S}}}\Lambda Q_{N}(du \times {\bf{Y}}\times d\zeta) \xrightarrow{{N \to \infty}} \int_{{\bf{U}} \times {\bf{S}}}\Lambda Q(du \times {\bf{Y}}\times d\zeta)} \label{eq:q2}.
\end{equation}
 Since the cost function $\tilde{c}$ is continuous in its last three arguments, ${P}$-almost surely
\begin{flalign}
& \scalemath{1}{\lim\limits_{N\rightarrow \infty}\tilde{c}\left(\omega_{0}, \zeta, u,\int_{\bf U}uQ_{N}(du \times {\bf{Y}}\times {\bf{S}}),\int_{{\bf{U}} \times {\bf{S}}}\Lambda Q_{N}(du \times {\bf{Y}}\times d\zeta)\right)}\nonumber\\
&~~~ \scalemath{1}{=\tilde{c}\left(\omega_{0}, \zeta, u,\int_{\bf U}uQ(du \times {\bf{Y}}\times {\bf{S}}),\int_{{\bf{U}} \times {\bf{S}}}\Lambda Q(du \times {\bf{Y}}\times d\zeta)\right)}\nonumber.
\end{flalign}
 Define a non-negative bounded functions
 \begin{flalign*}
 &\scalemath{1}{G_{N}^{M}:=\min\bigg\{M, }{\tilde{c}\left(\omega_{0}, \zeta, u,\int_{\bf U}uQ_{N}(du \times {\bf{Y}}\times {\bf{S}}),\int_{{\bf{U}} \times {\bf{S}}}\Lambda Q_{N}(du \times {\bf{Y}}\times d\zeta)\right)\bigg\}},\\
  &\scalemath{1}{G^{M}:=\min\bigg\{M, }{\tilde{c}\left(\omega_{0}, \zeta, u,\int_{\bf U}uQ(du \times {\bf{Y}}\times {\bf{S}}),\int_{{\bf{U}} \times {\bf{S}}}\Lambda Q(du \times {\bf{Y}}\times d\zeta)\right)\bigg\}},
 \end{flalign*}
where the sequence $\{G_{M}\}_{M}$ converges as $M \to \infty$ to 
\begin{equation*}
\scalemath{1}{G:=\tilde{c}\left(\omega_{0}, \zeta, u,\int_{\bf U}uQ(du \times {\bf{Y}}\times {\bf{S}}),\int_{{\bf{U}} \times {\bf{S}}}\Lambda Q(du \times {\bf{Y}}\times d\zeta)\right)}.
\end{equation*}
We have ${P}$-almost surely
\begin{flalign}
\scalemath{1}{\liminf\limits_{N\rightarrow \infty}}&\scalemath{0.94}{{\int_{\cal{Z}}\tilde{c}\left(\omega_{0}, \zeta, u,\int_{\bf U}uQ_{N}(du \times {\bf{Y}}\times {\bf{S}}),\int_{{\bf{U}} \times {\bf{S}}}\Lambda Q_{N}(du \times {\bf{Y}}\times d\zeta)\right)}\scalemath{1}{Q_{N}(du,dy,d\zeta)}}\nonumber\\
&\scalemath{1}{ \geq \lim\limits_{M\rightarrow \infty}\liminf\limits_{N\rightarrow \infty}\int_{\cal{Z}}G_{N}^{M}Q_{N}(du,dy.d\zeta)}\label{eq:bd1}\\
&\scalemath{1}{= \lim\limits_{M\rightarrow \infty}\int_{\cal{Z}}G^{M}Q(du,dy,d\zeta)}\label{eq:bd2}\\
&\scalemath{1}{=\int_{\cal{Z}}\tilde{c}\left(\omega_{0},\zeta, u,\int_{\bf U}uQ(du \times {\bf{Y}}\times {\bf{S}}),\int_{{\bf{U}} \times {\bf{S}}}\Lambda Q(du \times {\bf{Y}}\times d\zeta)\right)}\scalemath{1}{Q(du,dy,d\zeta)}\label{eq:bd3},
\end{flalign}
where \eqref{eq:bd1} is true since
\begin{flalign*}
\tilde{c}\left(\omega_{0}, \zeta, u,\int_{\bf U}uQ_{N}(du \times {\bf{Y}}\times {\bf{S}}),\int_{{\bf{U}} \times {\bf{S}}}\Lambda Q_{N}(du \times {\bf{Y}}\times d\zeta)\right)\geq G^{M}_{N}.
\end{flalign*}
Equality \eqref{eq:bd2} follows from the generalized convergence theorem in \cite[Theorem 3.5]{serfozo1982convergence} since $G_{N}^{M}$ is bounded and continuously converges to $G^{M}$, i.e., ${P}$-almost surely
\begin{flalign}
\lim\limits_{N\rightarrow \infty}&\min\bigg\{M,\scalemath{1}{~~~\tilde{c}\left(\omega_{0},\zeta, u_{N},\int_{\bf U}uQ_{N}(du \times {\bf{Y}}\times {\bf{S}}),\int_{{\bf{U}} \times {\bf{S}}}\Lambda Q_{N}(du \times {\bf{Y}}\times d\zeta)\right)\bigg\}}\nonumber\\
&\scalemath{1}{=\min\bigg\{M,}\scalemath{1}{~~~\tilde{c}\left(\omega_{0},\zeta, u,\int_{\bf U}uQ(du \times {\bf{Y}}\times {\bf{S}}),\int_{{\bf{U}} \times {\bf{S}}}\Lambda Q(du \times {\bf{Y}}\times d\zeta)\right)\bigg\}}\label{eq:contconv},
\end{flalign}
when $u_{N} \to u$ as $n \to \infty$. The monotone convergence theorem implies \eqref{eq:bd3}. 
Hence, \eqref{eq:2121} holds which implies $\limsup\limits_{N\rightarrow \infty}J_{T}^{N}(\underline{\tilde{\gamma}}_{T}^{*,N})=J_{T}^{\infty}(\underline{\tilde{\gamma}}_{T}^{*,\infty})$, and this completes the proof following from \cite[Theorem 5]{sanjari2018optimal}. Here, for completeness we present the proof which is similar to the analysis of the proof \cite[Theorem 5]{sanjari2018optimal} for dynamic teams,
\begin{flalign}
\inf\limits_{\underline{\gamma}_{T}}J^{\infty}_{T}(\underline{\gamma}_{T})&\leq \limsup\limits_{N\rightarrow \infty} J^{N}_{T}(\underline{\tilde{\gamma}}^{*,\infty}_{T})\nonumber\\
&=\limsup\limits_{N\rightarrow \infty}J_{T}^{N}(\underline{\gamma}_{T}^{*,N}) =\limsup\limits_{N\rightarrow \infty}\inf\limits_{\underline{\gamma}_{T}^{1:N}}J_{T}^{N}(\underline{\gamma}_{T}^{1:N})\nonumber\\
&= \limsup\limits_{N\rightarrow \infty}\inf\limits_{\underline{\gamma}_{T}}J_{T}^{N}(\underline{\gamma}_{T})\label{eq:34q}\\
&\leq \inf\limits_{\underline{\gamma}_{T}}\limsup\limits_{N\rightarrow \infty}J_{T}^{N}(\underline{\gamma}_{T}) =\inf\limits_{\underline{\gamma}_{T}}J^{\infty}(\underline{\gamma}_{T})\nonumber,
\end{flalign}
\end{itemize}
where \eqref{eq:34q} is true since the restriction $\underline{\gamma}_{T}$ to the first $N$ components is $\underline{\gamma}_{T}^{1:N}$. This implies that $\underline{\tilde{\gamma}}_{T}^{*,\infty}$ is globally optimal. 
\end{proof}
{\begin{remark}\label{rem:8}
On the connection between finitely many DMs and infinitely many DMs, we note a closely related work on mean-field games by Fischer \cite{fischer2017connection} where the information structure is assumed to be static since the policy of each player is assumed to be adapted to the filtration generated by his/her initial states and Wiener process (also called in the mean-field games' literature, somewhat non-standard in the control literature, as \textit{open-loop distributed} controllers \cite{lacker2016general},\cite[pages 72-76]{carmona2016mean}). This means that the information of each DM is not affected by any of the actions of the other DMs. For dynamic teams, there are two difficulties: (1) obtaining variational equations is challenging since fixing policies of DMs and perturbing only DM's policies, perturbs the observation of other DMs and hence the controls $u^{-i*}=(\gamma^{1}(y^{1}),\dots,\gamma^{i-1*}(y^{i-1}),\gamma^{i+1*}(y^{i+1}),\dots,\gamma^{N*}(y^{N}))$; (2) solutions of variational equations which give person-by-person optimal policies are inconclusive for global optimality due to the lack of convexity in general.
\end{remark}
\begin{remark}\label{rem:9}
We also note additional related works by Lacker \cite{lacker2017limit, lacker2018convergence} where either convergence of open-loop controllers, or convergence of Nash equilibria induced by closed-loop controllers (where controls are measurable path-dependent functions of states, $u_{t}^{i}=\phi(t,x_{0:t})$, where $x_{0:t}=(x_{0:t}^{1},\dots,x_{0:t}^{N})$ and $\phi$ is a measurable function) or Markovian controllers ($u_{t}^{i}=\phi(t,x_{t})$, where $x_{t}=(x_{t}^{1},\dots,x_{t}^{N})$) have been considered. In \cite{lacker2018convergence}, the information structure is classical (a centralized problem) since players have access to all the information available to previous DMs).
\end{remark}}
\begin{remark}
In Lemma \ref{lem:mflem} and Theorem \ref{the:mft}, we considered a non-classical information structure for teams defined as  ($\mathcal{P}_{T}^{\infty,\text{MF}}$) with a convex expected cost in policies. For teams defined as  ($\mathcal{P}_{T}^{\infty,\text{MF}}$) with a symmetric partially nested information structure which admit static reduction, the above result holds and similar to the proof of Theorem \ref{the:mft}, it can be proven under the assumption that the cost functions is convex in actions (since convexity of the cost function in actions is a sufficient condition for convexity of the expected cost function in policies for this class of problems \cite[Theorem 3.7]{YukselSaldiSICON17}).
\end{remark}
\begin{remark}\label{rem:mft} 
Assumptions that action spaces are compact and $f_{t}$s are bounded can be relaxed by assuming that
\begin{itemize}
\item [(A1)]$\sup\limits_{N\geq 1} E[|\underline{\gamma}^{*,N}_{T}(\underline{y}^{1})|^{1+\delta}]<\infty$  for some $\delta >0$,
\item [(A2)]$\sup\limits_{N\geq 1} E[|\Lambda(\underline{\gamma}^{*,N}_{T}(\underline{y}^{1}),\underline{\zeta}^{1})|^{1+\tilde{\delta}}]<\infty$  for some $\tilde{\delta} >0$.
\end{itemize}
That is because, following from the pointwise convergence of $\underline{\gamma}^{*,N}_{T}$ and continuity of $\Lambda$ in actions, the above uniform integrability assumption justifies exchanging the limit and the expectation required to establish the convergence in \eqref{eq:ucon} and \eqref{eq:q2} using a similar analysis of \eqref{eq:y199.5.2} and an argument of \eqref{eq:stlaw} based on the strong law of large numbers. This result is particularly important for LQG models (we use this remark in Section \ref{sec:LQG}).
\end{remark}

\begin{theorem}\label{the:mftergodic}
 Consider a team defined as  ($\mathcal{P}_{T}^{\infty,\text{MF}}$)  with ($\mathcal{P}_{T}^{N,\text{MF}}$) having a symmetric information structure for every $N$. Assume for every $N$ the team problem is convex in policies. Let the action space be compact and convex for each DM, and assume Assumption \ref{assump:c}, Assumption \ref{assump:2}(ii), and Assumption \ref{assump:ergodic} hold. If there exists a sequence of optimal policies for ($\mathcal{P}_{T}^{N,MF}$), $\{\underline{\gamma}^{*,N}_{T}\}_N$, which converges (for every DM due to the symmetry) pointwise to $\underline{\gamma}^{*,\infty}_{T}$ as $N \to \infty$, then $\underline{\gamma}^{*,\infty}_{T}$  (which is identically symmetric) is an optimal policy for ($\mathcal{P}_{T}^{\infty,MF}$).
\end{theorem}
\begin{proof}
Under Assumption  \ref{assump:2}(ii) and Assumption \ref{assump:ergodic}, for every $A^{i} \in \mathcal{B}({\bf{S}})$, and $A^{i}=B^{i}\times\prod_{t=0}^{T-1}(D_{t}^{i}\times E_{t}^{i})$ (where $B^{i} \in \mathcal{B}(\mathbb{X})$, $D_{t}^{i} \in \mathcal{B}(\mathbb{W})$, and $E_{t}^{i} \in \mathcal{B}(\mathbb{V})$), for all $N\in \mathbb{N}$, and permutations $\sigma$, we have $P$-almost surely
\begin{flalign}
\scalemath{1}{P(\underline{\zeta}^{1}\in A^{1},}&\scalemath{0.94}{\dots, \underline{\zeta}^{N}\in A^{N}| \omega_{0})}\nonumber\\
&\scalemath{1}{=P(x_{0}^{1}\in B^{1},\dots, x_{0}^{N}\in B^{N}| \omega_{0})\prod_{i=1}^{N}\prod_{t=0}^{T-1}P(w_{t}^{i}\in D_{t}^{i})P(v_{t}^{i}\in D_{t}^{i})}\label{eq:0ex1}\\
&\scalemath{1}{=P((x_{0}^{\sigma})^{1}\in B^{1},\dots, (x_{0}^{\sigma})^{N}\in B^{N}| \omega_{0})\prod_{i=1}^{N}\prod_{t=0}^{T-1}P((w_{t}^{\sigma})^{i}\in D_{t}^{i})P((v_{t}^{\sigma})^{i}\in D_{t}^{i})}\label{eq:0ex2}\\
&\scalemath{1}{=P((\underline{\zeta}^{\sigma})^{1}\in A^{1},\dots, (\underline{\zeta}^{\sigma})^{N}\in A^{N}| \omega_{0})}\nonumber,
\end{flalign}
where \eqref{eq:0ex1} follows from Assumption  \ref{assump:2}(ii), and \eqref{eq:0ex2} follows from Assumption \ref{assump:ergodic}. Hence, $(\underline{\zeta}^{1},\underline{\zeta}^{2},\dots)$ are exchangeable conditioned on $\omega_{0}$.

Hence, following from a similar argument as the proof of Theorem \ref{the:3.1} (by considering $\omega_{0}$ in the cost function and the law of total expectation (by first conditioning on $\omega_{0}$), under Assumption \ref{assump:2}(ii) and Assumption \ref{assump:ergodic}, one can consider a sequence of $N$-DM teams which are symmetrically optimal that defines ($\mathcal{P}_{T}^{N,\text{MF}}$) and whose limit is identified with $(\mathcal{P}_{T}^{\infty,\text{MF}})$. Since initial states are not necessarily independent conditioned on $\omega_{0}$, we can not establish that $(\underline{\gamma}^{*,\infty}_{T}(\underline{y}^{i}),\underline{y}^{i},\underline{\zeta}^{i})$ are i.i.d. random vectors conditioned on $\omega_{0}$ which has been used in \eqref{eq:stlaw} in (Step 1) of the proof of Theorem \ref{the:mft} to show that $Q_{N}$ converges weakly to $Q$ ${P}$-almost surely.

However, we note that since $(\underline{\zeta}^{1},\underline{\zeta}^{2},\dots)$ are exchangeable conditioned on $\omega_{0}$, for every $A^{i} \in \mathcal{B}({\bf{S}})$ and $C \in \mathcal{B}(\Omega_{0})$, and for all $N\in \mathbb{N}$, and permutations $\sigma$, we have
\begin{flalign}
{P(\underline{\zeta}^{1}\in A^{1},}&\scalemath{0.94}{\dots, \underline{\zeta}^{N}\in A^{N}, \omega_{0}\in C)}=\scalemath{0.94}{P((\underline{\zeta}^{\sigma})^{1}\in A^{1},}{\dots, (\underline{\zeta}^{\sigma})^{N}\in A^{N}, \omega_{0}\in C)}\label{eq:ome}.
\end{flalign}
Let $\alpha^{i}:=(\omega_{0},\underline{\zeta}^{i})$. Hence,  \eqref{eq:ome} implies that $(\alpha^{1},\alpha^{2},\dots)$ is exchangeable. Following from \cite[Proposition 3.8(a)]{aldous2006ecole}, there exists a random vector ${z}\in [0,1]$ such that, $({\underline{\zeta}}^{1},{\underline{\zeta}}^{2},\dots)$ are i.i.d. random vectors conditioned on $(\omega_{0},{z})$.

 Let $\tilde{\omega}_{0}:=(\omega_{0}, z)$. Hence, under Assumption  \ref{assump:2}(ii) and Assumption \ref{assump:ergodic}, conditioned on $\tilde{\omega}_{0}$, $(\underline{\zeta}^{1},\underline{\zeta}^{2},\dots)$ are i.i.d. random vectors. Following from standard stochastic realization results \cite[Lemma 3.1]{BorkarRealization}, we can represent any stochastic kernel in a functional form, with almost sure equivalence, $\underline{\zeta}^{i}=g(\tilde{\omega}_{0}, \theta^{i})$ for some independent $\theta^{i}$ and measurable $g$ (note that following from exchangeability, $g$ is identical for all $i\in \mathbb{N}$ and $(\theta^{1},\theta^{2},\dots)$ are i.i.d. random vectors).

Since conditioned on $\tilde{\omega}_{0}$, $(\underline{\zeta}^{1},\underline{\zeta}^{2},\dots)$ are i.i.d. random vectors, $(\underline{\gamma}^{*,\infty}_{T}(\underline{y}^{i}),\underline{y}^{i},\underline{\zeta}^{i})$ are i.i.d. random vectors conditioned on $\tilde{\omega}_{0}$, hence  for every $\tilde{\omega}_{0}$ on a set of $P$-measure one, we have for every continuous and bounded function $g$ in actions and observations, by the strong law of large numbers,
\begin{flalign*}
\scalemath{0.94}{P\left(\bigg\{\omega\in \Omega\bigg|\lim\limits_{N\rightarrow \infty}\left|\frac{1}{N}\sum_{i=1}^{N}g\bigg(\underline{\gamma}^{*,\infty}_{T}(\underline{y}^{i}),\underline{y}^{i},\underline{\zeta}^{i}\bigg)- E\bigg(g(\underline{\gamma}^{*,\infty}_{T}(\underline{y}^{1}),\underline{y}^{1},\underline{\zeta}^{1})\bigg|\tilde{\omega}_{0}\bigg)\right|=0\bigg\}\middle| \tilde{\omega}_{0}\right)=1}.
\end{flalign*}
Hence, following from an identical analysis as that of (Step 1) of the proof of Theorem \ref{the:mft}, conditioned on $\tilde{\omega}_{0}$, $Q_{N}$ converges weakly to $Q$, for every $\tilde{\omega}_{0}$ on a set of ${P}$-measure one. 

Following from the representation $\underline{\zeta}^{i}=g(\tilde{\omega}_{0}, \theta^{i})$, we have
\begin{flalign}
&\scalemath{1}{\limsup\limits_{N\rightarrow \infty} E\bigg[ E\bigg[}\scalemath{1}{\int_{\cal{Z}}\tilde{c}\left(\omega_{0},\zeta, u,\int_{\bf U}uQ_{N}(du \times {\bf{Y}}\times {\bf{S}}),\int_{{\bf{U}} \times {\bf{S}}}\Lambda Q_{N}(du \times {\bf{Y}}\times d\zeta)\right)}\nonumber\\
&~~~~~\:\:\:\:\:\:\:\:\:\:\:\:\:\:\:\scalemath{1}{\times Q_{N}(du,dy,d\zeta)\bigg|\omega_{0}\bigg]\bigg]}\nonumber\\
&\scalemath{1}{= \limsup\limits_{N\rightarrow \infty}\int_{{\Omega}_{0}\times [0,1]} E\bigg[}\scalemath{1}{\int_{\cal{Z}}\tilde{c}\left(\tilde{\omega}_{0},\zeta, u,\int_{\bf U}uQ_{N}(du \times {\bf{Y}}\times {\bf{S}}),\int_{{\bf{U}} \times {\bf{S}}}\Lambda Q_{N}(du \times {\bf{Y}}\times d\zeta)\right)}\label{eq:0ex31y}\\
&~~~~~\:\:\:\:\:\:\:\:\:\:\:\:\:\:\:\scalemath{1}{\times Q_{N}(du,dy,d\zeta)\bigg|\tilde{\omega}_{0}\bigg]P(d\tilde{\omega}_{0})}\nonumber,
\end{flalign}
where \eqref{eq:0ex31y} follows from the fact that for all $N \in \mathbb{N}$, and for every $A^{i} \in \mathcal{B}({\bf{S}})$
\begin{flalign*}
\scalemath{1}{\int _{\Omega_{0}}P(\underline{\zeta}^{1}\in A^{1}, \dots, \underline{\zeta}^{N}\in A^{N}|\omega_{0})P(d\omega_{0})=\int_{\Omega_{0}\times [0,1]}\prod_{i=1}^{N} \eta(\underline{\zeta}^{i}\in A^{i}|z, \omega_{0})P(dz,d\omega_{0})},
\end{flalign*}
and with slightly abuse of notations we use the same notation, $\tilde{c}$, for the cost function after transformation in \eqref{eq:0ex31y}. The rest of the proof is identical to that of Theorem \ref{the:mft}.
\end{proof}

\subsection{An existence theorem on globally optimal policies for dynamic mean-field team problems with a symmetric information structure}
An implication of Theorem \ref{the:mft} is the following existence result on globally optimal policies for mean-field team problems. In particular, we will establish the existence of a converging subsequence, in an appropriate sense, for a sequence of optimal policies for $N$-DM teams with an increasing number of DMs. For the following theorem, we do not establish the pointwise convergence; but by Theorem \ref{the:mft}, if a sequence of optimal policies for ($\mathcal{P}_{T}^{N,MF}$), $\{\underline{\gamma}^{*,N}_{T}\}_N$, converges pointwise, a global optimal policy exists. To this end, we allow decision makers to apply randomized policies.  For each decision maker (DM$^{i}$ for $i\in \mathbb{N}$), a probability measure $P \in \mathcal{P}(\Omega_{0}\times \mathbb{X} \times \prod_{t=0}^{T-1} (\mathbb{W}\times \mathbb{V})\times \prod_{t=0}^{T-1}(\mathbb{U}\times \mathbb{Y}))$ is a policy induced by a randomized policy if and only if for every $t=0,\dots,T-1$ and for all continuous and bounded function $g$
\begin{flalign}
&{\int g(\omega_{0}, x_{0}^{i}, \zeta_{0:t-1}^{i}, y_{0:t}^{i},u_{0:t}^{i})P(d\omega_{0}, dx_{0}^{i}, d\zeta_{0:t-1}^{i}, dy^{i}_{0:t},du^{i}_{0:t})}\nonumber\\
  &{=\int g(\omega_{0}, x_{0}^{i}, \zeta_{0:t-1}^{i},y_{0:t}^{i},u_{0:t}^{i})\mu^{i}(dx_{0}^{i}, d\zeta_{0:t-1}^{i}|\omega_{0})P(d\omega_{0})}\label{eq:eq18}\\
  &\:\:\:\:\:\:\ {\times \prod_{k=0}^{t} \Pi^{i}_{k}(du^{i}_{k}|y^{i}_{k})p_{k}^{i}(dy^{i}_{k}|\omega_{0}, x_{0}^{i}, \zeta_{0:k-1}^{i}, y^{i}_{0:k-1}, u_{0:k-1}^{i})}\nonumber,
\end{flalign}
for a stochastic kernel $\Pi^{i}_{k}$ on $\mathbb{U}$ given $\mathbb{Y}$, where $p_{k}^{i}$ is the transition kernel characterizing the observations of DM$^{i}$ at time $t$,
\begin{flalign*}
p_{k}^{i}&\left(y^{i}_{k}\in \cdot\middle|\omega_{0}, x_{0}^{i}, \zeta_{0:k-1}^{i}, y^{i}_{0:k-1}, u_{0:k-1}^{i}\right)\\
&:=P\left(h_{k}(x_{0:k}^{i},u_{0:k-1}^{i},v_{0:k}^{i}) \in \cdot\middle|\omega_{0}, x_{0}^{i}, \zeta_{0:k-1}^{i}, y^{i}_{0:k-1}, u_{0:k-1}^{i}\right),
\end{flalign*}
and $\mu^{i}$ is a fixed probability measure on initial states and disturbances of DM$^{i}$ conditioned on $\omega_{0}$. This equivalency follows from the fact that continuous and bounded functions form a separating class \cite[page 12]{Billingsley} and \cite[Theorem 2.2]{yukselSICON2017}. 

First, we present an absolute continuity assumption on observations of DMs.
{\begin{assumption}\label{assump:a4.2}
For every DM$^{i}$ and $t=0,\dots,T-1$, there exists a function $\psi_{t}^{i}:\mathbb{Y}\times \Omega_{0}\times \mathbb{X}\times \prod_{k=0}^{t-1}(\mathbb{W}\times \mathbb{V})\times\prod_{k=0}^{t-1}(\mathbb{Y}\times\mathbb{U})\to \mathbb{R}_{+}$ continuous in actions, and a probability measure $\nu_{t}^{i}$ on $\mathbb{Y}$ such that  for all Borel sets $A=A^{1}\times \dots \times A^{N}$,
\begin{flalign*}
P(y_{t}^{1:N}\in A&|\omega_{0}, x_{0}^{1:N}, \zeta_{0:t-1}^{1:N}, y^{1:N}_{0:t-1}, u_{0:t-1}^{1:N})\\
&{=\prod_{i=1}^{N}\int_{A^i}\psi_{t}^{i}(y^{i}_{t}, \omega_{0}, x_{0}^{i},\zeta_{0:t-1}^{i},y_{0:t-1}^{i}, u_{0:t-1}^{i})\nu_{t}^{i}(dy_{t}^{i})}.
\end{flalign*}
\end{assumption}}

This assumption allows us to obtain an independent measurements reduction (see \cite[Section 2.2]{yuksel2018general}). For example, if $v_{t}^{i}$ for all $i \in \mathbb{N}$ and $t=0,\dots,T-1$ are i.i.d with a probability measure admitting a density function so that the observation of each DM$^{i}$ at time $t$ is
$y_{t}^{i}=\tilde{h}_{t}(x_{t}^{i},u_{\downarrow t}^{i})+v_{t}^{i}$,
where $\tilde{h}_{t}$ is continuous, then Assumption \ref{assump:a4.2} holds \cite[Lemma 5.1]{gupta2014existence}.

\begin{theorem}\label{the:exmft}
Consider a team defined as  ($\mathcal{P}_{T}^{\infty,\text{MF}}$) with ($\mathcal{P}_{T}^{N,\text{MF}}$) having a symmetric information structure for every $N$. Assume for every $N$, the team problem is convex in policies and the action space is convex. Assume further that without any loss, the optimal policies can be restricted to those with $ E(\phi_{i}({u}^{i}))\leq K$ for some finite $K$, where $\phi_{i}:{\mathbb{U}} \to \mathbb{R}_{+}$ is lower semi-continuous (moment condition). Under Assumption \ref{assump:c} and Assumption \ref{assump:2} if either
\begin{itemize}
\item [(i)] Assumption \ref{assump:a4.2} holds (with no further assumptions on the information structure of each DM$^{i}$ for $i\in \mathbb{N}$ through time $t=0,\dots, T-1$), or
\item [(ii)] for each DM$^{i}$ for $i\in \mathbb{N}$ through time $t=0,\dots, T-1$, there exists a static reduction with the classical information structure (i.e., under a static reduction, the information structure is expanding such that $\sigma(y_{t}^{i}) \subset \sigma(y_{t+1}^{i})$ where $\sigma$ denotes the $\sigma$-field),
\end{itemize}
then there exists an optimal policy for ($\mathcal{P}_{T}^{\infty,MF}$).
\end{theorem}

Since the space of policies that are deterministic (where $\Pi_{k}^{i}$ in \eqref{eq:eq18} are indicator functions) are not closed under the weak convergence topology (e.g., as an implication of \cite[Theorem 2.7]{YukselSaldiSICON17}), we allow for randomization in the policies and therefore the limit policy is not necessarily deterministic according to the above result; however, it is identical for each DM.
\begin{proof}
We use individually randomized policies and we show that for every sequence of $N$-DM optimal policies, there exists a subsequence which converges to an optimal independently randomized policy for the mean-field limit under an appropriate topology defined by the product topology where each coordinate is endowed with the weak convergence topology. In (Step 1), we show that for each finite $N$-DM team problem, optimal policies are deterministic and symmetric and we consider the independently randomized policies induced by such policies $\{P_{N}\}_{N}$ (where $P_{N} \in\mathcal{P}({\bf{Y}} \times {\bf{U}})$ for each DM satisfying \eqref{eq:eq18}), as our sequence to be studied. We also define the sequence of empirical measures  induced by these policies, $Q_{N}$, as \eqref{eq:q}.

In (Step 2), we show that for every sequence of policies satisfying a moment condition there exists a subsequence such that policies $\{P_{n}\}_{n}$ for each DM and a subsequence of empirical measures $\{Q_{n}\}_{n}$ induced by these policies (where $n \in \mathbb{I}$ is the index set of a convergent subsequence) converge weakly to a limit $P$-almost surely, that is, for a set of ${P}$-measure one, for every bounded function $g$ which is continuous in actions and observations and measurable in uncertainties,
\begin{flalign*}
{P\bigg(\bigg\{\omega\in \Omega\bigg|\lim\limits_{n\rightarrow \infty}\left(\int g dQ_{n}-\int g dQ\right)=0\bigg\}\bigg| \omega_{0}\bigg)=1}.
\end{flalign*}
To this end, we first show that for each DM a sequence $\{P_{n}\}_{n}$ is tight, then we show that the sequence of empirical measures $\{Q_{n}\}_{n}$ induced by these policies converges weakly to a limit $P$-almost surely.  

In (Step 3), we show that the set of policies for each DM is closed under the weak convergence topology, hence, the limit policy satisfies the required measurability/conditional independence constraints (that is, the limit policy satisfies \eqref{eq:eq18}). In (Step 4), we use the lower semi-continuity argument to show that the expected cost function under the induced limit policy is less than or equal to the expected cost achieved by the sequence of $N$-DM optimal policies. 
\begin{itemize}[wide]
\item [\textbf{(Step 1):}]
Under Assumption \ref{assump:c}, Assumption \ref{assump:c}(c) and Assumption \ref{assump:2}, and by condition (i) using \cite[Theorem 5.2]{yuksel2018general}, or condition (ii) using \cite[Theorem 5.6]{yuksel2018general}, there exists a deterministic optimal policy for each finite $N$-DM team problem.  
Action spaces are convex and the team problem is convex in policies, hence, using Lemma \ref{lem:mflem}, one can consider a sequence of $N$-DM teams which are symmetrically optimal that defines ($\mathcal{P}_{T}^{N,MF}$) and whose limit is identified with ($\mathcal{P}_{T}^{\infty,MF}$). Hence, for each $N$-DM team problem, we consider symmetric randomized optimal policies. 
\item [\textbf{(Step 2):}]
In the following, we first show that the set of policies $P_{N} \in\mathcal{P}({\bf{Y}} \times {\bf{U}})$ for each DM satisfying \eqref{eq:eq18} and the moment condition is tight, then, by symmetry, we show that $\{Q_{N}\}_{N}$ is induced by this set of policies is also tight. We use the fact that  conditioned on $\omega_{0}$,  $(\underline{\gamma}^{*,N}_{T}(\underline{y}^{i}),\underline{y}^{i},\underline{\zeta}^{i})$ are i.i.d. random vectors (this follows from symmetry of the information structure and Lemma \ref{lem:mflem} since every DM applies the identical policy $\underline{\gamma}_{T}^{*,N}$) and also since the space of control policies is tight under the weak convergence for each DM (see e.g., \cite[proof of Theorem 4.7]{yuksel2018general}). 

Since actions of DMs do not affect the observations of others, the policy spaces are decoupled from the actions of other decision makers. Since we can restrict the search for optimal policies over those satisfying the moment condition, the fact that $\nu \to \int \nu(dx)g(x)$ is lower semi-continuous for a continuous function $g$ \cite[proof of Theorem 4.7]{yuksel2018general} implies that the marginals on $\bf{U}$ satisfying the moment condition are tight under the weak convergence topology. Hence, the collection of all probability measures with these tight marginals is also tight (see e.g., \cite[Proof of Theorem 2.4]{yukselSICON2017}). This implies that the sequence of randomized policies satisfying the moment condition is tight.  

Since every DM applies an identical policy and since observations are conditionally i.i.d., a countably infinite product of space of policies of DMs is tight (where each coordinate is tight in the weak convergence topology). Hence,  there exists a subsequence of policies $\tilde{P}_{n}\in \mathcal{P}(\prod_{i}({\bf{Y}} \times {\bf{U}}))$ (as a product of policies of DMs) converges weakly  to a limit $\tilde{P}$ (each coordinate converges weakly)  ${P}$-almost surely. Furthermore, since every DM applies an identical policy, conditioned on $\omega_{0}$, actions induced by an identical randomized policies, observations and disturbances are i.i.d. through DMs. Hence, following from a similar argument as (Step 1) of the proof of Theorem \ref{the:mft}, a subsequence of empirical measures $\{Q_{n}\}_{n \in \mathbb{I}}$ converges ${P}$-almost surely to $Q$ in $w$-$s$ topology. We note that the convergence is under the weak convergence topology, but since $\underline{\zeta}^{i}$s are exogenous with a fixed marginal, the convergence is also in the $w$-$s$ topology.
\item [\textbf{(Step 3):}] 
In this step, we show that each coordinate of the space of policies (space of policies for each DM) is closed under the weak convergence topology. This in particular implies that the space of policies is closed under the product topology and using (Step 1), we can conclude that the space of control policies is compact under the product topology where each coordinate is weakly compact.

Assume $P_{n}$ is a policy for DM$^{i}$ induced by a randomized policy converging weakly to $P_\infty$. In fact, conditions (i) or (ii) leads to the closedness of the set of policies (see \eqref{eq:eq18}) induced by $P_{n}$.
If Assumption \ref{assump:a4.2} holds, then  by the discussion in the proof of \cite[Theorem 5.2]{yuksel2018general}, each coordinate of policy spaces corresponding to DM$^{i}$ acting through time is closed under the weak convergence topology. Also, if condition (ii) holds, then \cite[Theorem 5.6]{yuksel2018general} leads to the same conclusion. Hence, each coordinate of space of policies (corresponding to DM$^{i}$) is closed under the weak convergence topology (since each coordinate of the space of policies is a finite product of space of policies for each DM at time instances $t=0,\dots,T-1$). Hence, following from (Step 2), there exists a subsequence $\{Q_{n}\}_{n \in \mathbb{I}}$ converges weakly to $Q$ ${P}$-almost surely where $Q$ is induced by a randomized policy in the set of policies satisfying \eqref{eq:eq18} and the limit policy is admissible and satisfies the required measurability/conditional independence constraints.

{Let for every $t=0,\dots,T-1$, $P^{*, \omega_{0}}_{n}$ be a probability measure on actions, observations and uncertainties induced by optimal randomized policies for each DM (which is identical because of symmetry) for $N$-DM teams conditioned on $\omega_{0}$, i.e., a probability measure that satisfies
\begin{flalign}
&{\int g(\omega_{0}, x_{0}^{i}, \zeta_{0:t-1}^{i}, y_{0:t}^{i},u_{n, 0:t}^{i, *})P^{*, \omega_{0}}_{n}(dx_{0}^{i}, d\zeta_{0:t-1}^{i}, dy^{i}_{0:t},du^{i, *}_{n, 0:t})}\nonumber\\
  &{=\int g(\omega_{0}, x_{0}^{i}, \zeta_{0:t-1}^{i},y_{0:t}^{i},u_{n, 0:t}^{i, *})\mu^{i}(dx_{0}^{i}, d\zeta_{0:t-1}^{i}|\omega_{0})}\label{eq:eq18dyn}\\
  &\:\:\:\:\:\:{ \times \prod_{k=0}^{t} \Pi_{k}^{*, n}(du^{*, i}_{n, k}|y^{i}_{k})p_{k}(dy^{i}_{k}|\omega_{0}, x_{0}^{i}, \zeta_{0:k-1}^{i}, y^{i}_{0:k-1}, u_{n, 0:k-1}^{i, *})}\nonumber,
\end{flalign}
for all bounded functions $g$ which is continuous in actions and observations and measurable in other arguments. We denote $\underline{u}^{i,*}_{n}:=({u}^{i,*}_{n, 0},\dots,{u}^{i,*}_{n, T-1})$ as the action of DM$^{i}$ through time induced by $\Pi^{*,n}_{t}$. Similarly, we denote $P^{*, \omega_{0}}$ as a probability measure induced by the limit policy, i.e., a probability measure satisfying \eqref{eq:eq18dyn} induced by $\Pi_{k}^{*, \infty}$ where $\underline{u}^{i,*}_{\infty}:=({u}^{i,*}_{\infty, 0},\dots,{u}^{i,*}_{\infty, T-1})$ is the action of DM$^{i}$ through time induced by $\Pi_{k}^{*, \infty}$. }
\item [\textbf{(Step 4):}] 
Now, we show that the expected cost function under the limit randomized policy is less than or equal to the expected cost achieved by $\limsup\limits_{n\rightarrow \infty}J_{T}^{n}(\underline{\tilde{\gamma}}_{T}^{*,n})$. {Since the cost function is continuous in states and actions, under the reduction (conditions (i) or (ii)), we have ${P}$-almost surely
\begin{flalign}
\scalemath{1}{\frac{1}{N}\sum_{i=1}^{N}\sum_{t=0}^{T-1}}&\scalemath{1}{c\bigg(\omega_{0}, x_{t}^{i},u_{t}^{i},\frac{1}{N}\sum_{p=1}^{N}u_{t}^{p},\frac{1}{N}\sum_{p=1}^{N}x_{t}^{p}\bigg)\prod_{i=1}^{N}\prod_{t=0}^{T-1}{{\psi}_{t}}\bigg({y}^{i}_{t}, \omega_{0}, x_{0}^{i},{\zeta}^{i}_{0:t-1},{y}_{0:t-1}^{i},{u}_{0:t-1}^{i}\bigg)}\nonumber\\
&\scalemath{1}{=\frac{1}{N}\sum_{i=1}^{N}\bar{c}\bigg(\omega_{0},\underline{\zeta}^{i}, \underline{u}^{i},\frac{1}{N}\sum_{p=1}^{N}\underline{u}^{p},\frac{1}{N}\sum_{p=1}^{N}\Lambda(\underline{u}^{p},\underline{\zeta}^{p})\bigg)\prod_{i=1}^{N}{{\underline\psi}}\bigg(\underline{y}^{i}, \omega_{0},\underline{\zeta}^{i},\underline{u}^{i}\bigg)}\label{eq:star12di},
\end{flalign}
where \eqref{eq:star12di} is true following from \eqref{eq:mfobs} and Assumption \ref{assump:c} for some functions $\bar{c}:\Omega_{0}\times {\bf S} \times {\bf U} \times {\bf U} \times {\bf X} \to \mathbb{R}_{+}$  continuous in states and actions and a function $\Lambda: {\bf U} \times {\bf S} \to {\bf X}$ continuous in actions and 
\begin{flalign*}
&\scalemath{1}{\prod_{i=1}^{N}{{\underline\psi}}\bigg(\underline{y}^{i}, \omega_{0},\underline{\zeta}^{i},\underline{u}^{i}\bigg):=\prod_{i=1}^{N}\prod_{t=0}^{T-1}{{\psi}_{t}}\bigg({y}^{i}_{t}, \omega_{0}, x_{0}^{i},{\zeta}^{i}_{0:t-1},{y}_{0:t-1}^{i},{u}_{0:t-1}^{i}\bigg).}
\end{flalign*}} 
{We have
\begin{flalign}
&\scalemath{1}{\limsup\limits_{N\rightarrow \infty}\frac{1}{N}\sum_{i=1}^{N}\sum_{t=0}^{T-1} E^{\underline{\gamma}_{T}^{*,1:N}}\left[c\left(\omega_{0}, x_{t}^{i},u_{t}^{i},\frac{1}{N}\sum_{p=1}^{N}u_{t}^{p},\frac{1}{N}\sum_{p=1}^{N}x_{t}^{p}\right)\right]}\nonumber\\
&\scalemath{0.94}{ \geq  \lim\limits_{M\rightarrow \infty}\limsup\limits_{N\rightarrow \infty}\int\int_{\cal{Z}}\min\bigg\{M,} \scalemath{0.94}{\tilde{c}\left(\omega_{0},\zeta, u,\int_{\bf U}uQ_{N}(du \times {\bf{Y}}\times {\bf{S}}),\int_{{\bf{U}} \times {\bf{S}}}\Lambda Q_{N}(du \times {\bf{Y}}\times d\zeta)\right)\bigg\}}\label{eq:36.5.1}\\
&~~~~~\:\:\:\:\:\:\:\:\:\:\:\:\:\:\:\scalemath{0.94}{\times Q_{N}(du,dy,d\zeta)\prod_{i=1}^{\infty}P_{N}^{*,\omega_{0}}(d\underline{u}^{i, *}_{N},d\underline{y}^{i},d\underline{\zeta}^{i})\prod_{i=1}^{\infty}{{\underline\psi}}(\underline{y}^{i}, \omega_{0},\underline{\zeta}^{i},\underline{u}^{i,*}_{N}){P}(d\omega_{0})}\nonumber\\
&\scalemath{0.94}{\geq  \lim\limits_{M\rightarrow \infty}\lim\limits_{n\rightarrow \infty}\int\int_{\cal{Z}}\min\bigg\{M,}\scalemath{0.94}{\tilde{c}\left(\omega_{0},\zeta, u,\int_{\bf U}uQ_{n}(du \times {\bf{Y}}\times {\bf{S}}),\int_{{\bf{U}} \times {\bf{S}}}\Lambda Q_{n}(du \times {\bf{Y}}\times d\zeta)\right)\bigg\}}\label{eq:35.5.1}\\
&~~~~~\:\:\:\:\:\:\:\:\:\:\:\:\:\:\:\scalemath{0.94}{\times Q_{n}(du,dy,d\zeta)\prod_{i=1}^{\infty}P_{n}^{*,\omega_{0}}(d\underline{u}^{i, *}_{n},d\underline{y}^{i},d\underline{\zeta}^{i})\prod_{i=1}^{\infty}{{\underline\psi}}(\underline{y}^{i}, \omega_{0},\underline{\zeta}^{i},\underline{u}^{i,*}_{n}){P}(d\omega_{0})}\nonumber\\
&\scalemath{0.94}{=  \lim\limits_{M\rightarrow \infty}\int\lim\limits_{n\rightarrow \infty}\int_{\cal{Z}}\min\bigg\{M,}\scalemath{0.94}{\tilde{c}\left(\omega_{0},\zeta, u,\int_{\bf U}uQ_{n}(du \times {\bf{Y}}\times {\bf{S}}),\int_{{\bf{U}} \times {\bf{S}}}\Lambda Q_{n}(du \times {\bf{Y}}\times d\zeta)\right)\bigg\}}\label{eq:34.5.1}\\
&~~~~~\:\:\:\:\:\:\:\:\:\:\:\:\:\:\:\scalemath{0.94}{\times  Q_{n}(du,dy,d\zeta)\prod_{i=1}^{\infty}P_{n}^{*,\omega_{0}}(d\underline{u}^{i, *}_{n},d\underline{y}^{i},d\underline{\zeta}^{i})\prod_{i=1}^{\infty}{{\underline\psi}}(\underline{y}^{i}, \omega_{0},\underline{\zeta}^{i},\underline{u}^{i,*}_{n}){P}(d\omega_{0})}\nonumber\\
&\scalemath{1}{=\lim\limits_{M\rightarrow \infty}\int\int_{\cal{Z}}\min\bigg\{M,}{\tilde{c}\left(\omega_{0},\zeta, u,\int_{\bf U}uQ(du \times {\bf{Y}}\times {\bf{S}}),\int_{{\bf{U}} \times {\bf{S}}}\Lambda Q(du \times {\bf{Y}}\times d\zeta)\right)\bigg\}}\label{eq:33.5.1}\\
&~~~~~\:\:\:\:\:\:\:\:\:\:\:\:\:\:\:\scalemath{1}{\times Q(du,dy,d\zeta)\prod_{i=1}^{\infty}P^{*,\omega_{0}}(d\underline{u}^{i, *}_{\infty},d\underline{y}^{i},d\underline{\zeta}^{i})\prod_{i=1}^{\infty}{{\underline\psi}}(\underline{y}^{i}, \omega_{0},\underline{\zeta}^{i},\underline{u}^{i,*}_{\infty}){P}(d\omega_{0})}\nonumber\\
&\scalemath{1}{=\int\int_{\cal{Z}}\tilde{c}\left(\omega_{0},\zeta, u,\int_{\bf U}uQ(du \times {\bf{Y}}\times {\bf{S}}),\int_{{\bf{U}} \times {\bf{S}}}\Lambda Q(du \times {\bf{Y}}\times d\zeta)\right)}\label{eq:37.5.1}\\
&~~~~~\:\:\:\:\:\:\:\:\:\:\scalemath{1}{\times Q(du,dy,d\zeta)\prod_{i=1}^{\infty}P^{*,\omega_{0}}(d\underline{u}^{i, *}_{\infty},d\underline{y}^{i},d\underline{\zeta}^{i})\prod_{i=1}^{\infty}{{\underline\psi}}(\underline{y}^{i}, \omega_{0},\underline{\zeta}^{i},\underline{u}^{i,*}_{\infty}){P}(d\omega_{0})}\nonumber
\end{flalign}
}where \eqref{eq:36.5.1} follows from the definition of empirical measures and by integrating over the set $(\prod_{i=n_{l}+1}^{\infty}{\bf{Y}\times {\bf{S}}})$ and since ${P}$-almost surely
\begin{flalign*}
&\scalemath{1}{\min\bigg\{M,\tilde{c}\left(\omega_{0},\zeta, u,\int_{\bf U}uQ_{N}(du \times {\bf{Y}}\times {\bf{S}}),\int_{{\bf{U}} \times {\bf{S}}}\Lambda Q_{N}(du \times {\bf{Y}}\times d\zeta)\right)\bigg\}}\\
&\scalemath{1}{\leq \tilde{c}\left(\omega_{0},\zeta, u,\int_{\bf U}uQ_{N}(du \times {\bf{Y}}\times {\bf{S}}),\int_{{\bf{U}} \times {\bf{S}}}\Lambda Q_{N}(du \times {\bf{Y}}\times d\zeta)\right)}.
\end{flalign*}
Inequality \eqref{eq:35.5.1} is true since limsup is the greatest convergent subsequential limit for a bounded sequence and \eqref{eq:34.5.1} follows from the dominated convergence theorem. We note that $\{Q_{n}\}_{n}$ is induced by $\underline{u}^{i,*}_{n}$ for each DM. Since $\{Q_{n}\}_{n \in I}$ converges weakly to $Q$ ${P}$-almost surely, by the moment condition and Remark \ref{rem:mft}, a similar argument as (Step 3) of the proof of Theorem \ref{the:mft} implies that ${P}$-almost surely
\begin{flalign*}
&\scalemath{1}{\int_{\bf U}uQ_{n}(du \times {\bf{Y}} \times {\bf{S}}) \xrightarrow{{n \to \infty}} \int_{\bf U}uQ(du \times {\bf{Y}}\times {\bf{S}})}\\
&\scalemath{1}{\int_{{\bf{U}} \times {\bf{S}}}\Lambda Q_{n}(du \times {\bf{Y}}\times d\zeta) \xrightarrow{{n \to \infty}} \int_{{\bf{U}} \times {\bf{S}}}\Lambda Q(du \times {\bf{Y}}\times d\zeta)}
\end{flalign*}
Hence, \eqref{eq:33.5.1} follows from \cite[Theorem 3.5]{serfozo1982convergence} since
\begin{flalign*}
&\scalemath{1}{\min\bigg\{M,}\scalemath{0.94}{\tilde{c}\left(\omega_{0},\zeta, u,\int_{\bf U}uQ_{n}(du \times {\bf{Y}}\times {\bf{S}}),\int_{{\bf{U}} \times {\bf{S}}}\Lambda Q_{n}(du \times {\bf{Y}}\times d\zeta)\right)\bigg\}}
\end{flalign*}
is bounded and non-negative, and continuously converges in $u$ ${P}$-almost surely (see \eqref{eq:contconv}). That is because, conditioned on $\omega_{0}$, $\underline{y}^{i}$ are i.i.d. random vectors (thanks to the symmetry), the space of policies is compact under the product topology (with the weak convergence topology for each coordinate (for each DM)), $\prod_{i=1}^{\infty}{{\underline\psi}}(\underline{y}^{i}, \omega_{0},\underline{\zeta}^{i},\underline{u}^{i,*}_{n})$ converges in the product topology, and the cost function and $\underline\psi$ are continuous. Finally, \eqref{eq:37.5.1} follows from the monotone convergence theorem.
Hence, the proof is completed.
\end{itemize}

\end{proof}
\begin{remark}
For the existence result, to show that the set of policies induced by independently randomized policies for each DM through time $t=0,\dots, T-1$ (see \eqref{eq:eq18}) is closed under the weak convergence topology, we utilized the result in \cite[Section 5.2]{yuksel2018general} which are more general than those in \cite{gupta2014existence, YukselSaldiSICON17}. We note that the extension of the existence results in  \cite[Section 5.2]{yuksel2018general} to our setup is not immediate since the conclusion of (Step 3) can not be established rigorously without considering the technical steps involving infinite dimensions and limit arguments. 
\end{remark}

\section{Symmetric LQG dynamic teams}\label{sec:LQG}
In the section, we consider LQG setup where the results of Section \ref{sec:4} and Section \ref{sec:mft} can be applied. {We first consider $N$-DM LQG problems  where we use Theorem \ref{the:3.1} to show that the globally optimal policies are symmetric. Then, based on symmetry, we calculate $N$-DM optimal policies for such problems. Next, using Theorem  \ref{the:mft} and Theorem \ref{the:mftergodic}, we show the convergence of $N$-DM optimal policies to optimal policies of LQG mean-field teams with countably infinite number of DMs. Finally, we consider infinite horizon problems where we use symmetry and convergence results to obtain global optimal policies for such problems. }
\subsection{Symmetric partially nested LQG dynamic teams on a graph}\label{sec:3a}
In the following, we consider decentralized problems where Theorem \ref{the:3.1} can be utilized and the optimal policy can be obtained. First, we formulate LQG problems with a symmetric partially nested information structure.
Consider the following dynamics. Let $i=1,2$, and
\begin{equation}\label{eq:state}
x_{t+1}^{i}=Ax_{t}^{i}+Bu_{t}^{i}+w^{i}_{t}.
\end{equation}
\begin{itemize}[wide = 0pt]
\item[\textbf{Problem} \bf{($\mathcal{P}_{T}$)}:]
Consider the expected cost function of $(\underline{\gamma}_{T}^{1},\underline{\gamma}_{T}^{2})$ as 
\begin{flalign}\label{eq:1.1}
\scalemath{1}{J_{T}(\underline{\gamma}_{T}^{1},\underline{\gamma}_{T}^{2})= E^{(\underline{\gamma}_{T}^{1},\underline{\gamma}_{T}^{2})}\bigg[\frac{1}{T}\sum_{t=0}^{T-1}\sum_{i=1}^{2}(x_{t}^{i})^{T}Qx_{t}^{i}+(u_{t}^{i})^{T}Ru_{t}^{i}+(u_{t}^{1})^{T}\tilde Ru_{t}^{2}+(u_{t}^{2})^{T}\tilde Ru_{t}^{1}\bigg]},
\end{flalign}
where $\underline{\gamma}_{T}^{i}=(\gamma^{i}_{0:T-1})$, and $R,\tilde R>0$ and $Q \geq 0$. {Let
\begin{flalign}
\scalemath{1}{{y}_{t}^{i}=H_{t}\zeta^{i}_{0:t-1}+\sum_{j=0}^{t-1}D_{tj}u_{j}^{i}}\label{eq:observationlqg},
 \end{flalign}
 where ${\zeta}^{i}_{t}=(w_{t}^{i}, v_{t}^{i})$ with ${\zeta}^{i}_{0}=(x_{0}^{i},w_{0}^{i}, v_{0}^{i})$. Let $n, m, s \in \mathbb{N}$ and $\mathbb{X}=\mathbb{R}^{n}$, $\mathbb{Y}=\mathbb{R}^{s}$, $\mathbb{U}=\mathbb{R}^{m}$, $w_{t}^{i} \in \mathbb{R}^{n}$, $v_{t}^{i} \in \mathbb{R}^{n}$,$A \in \mathcal{M}_{n,n}$,   $B \in \mathcal{M}_{n,m}$, $R \in \mathcal{M}_{m,m}$, $Q \in \mathcal{M}_{n,n}$, $\tilde{R} \in \mathcal{M}_{m,m}$, $H_{t} \in \mathcal{M}_{s,n(2t+1)}$, and $D_{tj} \in \mathcal{M}_{s,m}$. Let the information structure of each DM be $I_{t}^{i}=\{{y}^{i}_{t}, {y}_{\downarrow t}^{i}\}$.}
\end{itemize}

In the following, we show that the above dynamic teams are symmetrically optimal under sufficient conditions on the observations and initial states. 
\begin{corollary}\label{lem:4.4.4.4}
For a fixed $T$, consider a finite horizon team problem defined above as ($\mathcal{P}_{T}$). If $x_{0}^{1}$ and $x_{0}^{2}$ are exchangeable zero mean Gaussian random vectors and $w^{i}_{t}$s and $v^{i}_{t}$s are i.i.d. Gaussian random vectors for $i=1,2$ and independent for all $t=0,\dots, T-1$ and also independent of initial states, then the dynamic team is symmetrically optimal.
\end{corollary}
\begin{proof}
Since the dynamic team is LQG with a partially nested information structure, a static reduction exists and the expected cost is convex in policies under static reductions (see \cite{HoChu} and Theorem \ref{the:cp}(iii)). Assumption \ref{assump:3.1} is satisfied following from \eqref{eq:1.1}. We need to show assumptions of Theorem \ref{the:3.1} hold. Following from the hypothesis on disturbances and initial states, Assumption (b) holds. Assumption (c) holds following from Assumption \ref{the:3.1} and since given $(x_{0}^{1}, x_{0}^{2})$, $(y_{0:T-1}^{1}, y_{0:T-1}^{2})$ are independent. Hence, Theorem \ref{the:3.1} completes the proof.
\end{proof}

Here, we consider a class of LQG dynamic teams with a tree information structure where we utilize Corollary \ref{lem:4.4.4.4} and we obtain an explicit recursion for the optimal policy.
\begin{itemize}[wide = 0pt]
\item[\textbf{Problem} \bf{($\mathcal{P}_{T}^{\text{tree}}$)}:]
Consider a finite horizon expected cost \eqref{eq:1.1} with $I_{t}^{i}=\{x_{[0:t]}^{i}, u_{[0:t-1]}^{i}\}$.
\end{itemize}

{We note that this problem is a special case of ($\mathcal{P}_{T}$) since we assumed that $y_{t}^{i}=x_{t}^{i}$ for all DM$^{i}$ and $t=0,\dots,T-1$ and also all DMs have a total recall ($I_{t}^{i}=\{x_{[0:t]}^{i}, u_{[0:t-1]}^{i}\}$). For this problem, we calculate an explicit recursion for optimal policies using the symmetry established in Corollary \ref{lem:4.4.4.4}.}

 \begin{theorem}\label{the:4.2}
 For a fixed $T$, consider a finite horizon team problem defined as ($\mathcal{P}_{T}^{\text{tree}}$). If ($x_{0}^{1},x_{0}^{2})$  are exchangeable with an identical zero mean Gaussian distribution and $w^{i}_{t}$s are i.i.d. zero mean Gaussian random vectors for $i=1,2$ and independent for all $t=0,\dots, T-1$ and independent of initial states, then
 \begin{flalign}
 &u^{1*,(T)}_{t}=K_{t}^{(T)}x_{t}^{1}+L_{t}^{(T)} E[x_{0}^{2}|x_{0}^{1}]\label{eq:4.9},\\
 &u^{2*,(T)}_{t}=K_{t}^{(T)}x_{t}^{2}+L_{t}^{(T)} E[x_{0}^{1}|x_{0}^{2}]\label{eq:4.10},
 \end{flalign}
 where
 \begin{flalign}
& K_{t}^{(T)}=-(R+B^{T}P_{t+1}^{(T)}B)^{-1}B^{T}P_{t+1}^{(T)}A\label{eq:3.4.15},\\
 &P_{t}^{(T)}=-A^{T}P_{t+1}^{(T)}B^{T}(R+B^{T}P_{t+1}^{(T)}B)^{-1}B^{T}P_{t+1}^{(T)}A+Q+A^{T}P_{t+1}^{(T)}A\label{eq:3.4.16},\\
 &L_{t}^{(T)}=-\bigg(R+B^{T}P_{t+1}^{(T)}B\bigg)^{-1}\bigg[\tilde{R}K_{t}^{(T)}G_{t}^{(T)}+\tilde{R}L_{t}^{(T)}\Sigma\label{eq:3.4.17}\\
 &~~~~\:\:\:\:\:\:\:\:\:\:\:\:+\sum_{s=t+1}^{T-1}B^{T}(A^{T})^{s-t}P_{s+1}^{(T)}BL_{s}^{(T)}\bigg]\nonumber,\\
& G_{t}^{(T)}=\prod_{s=0}^{t-1}(A+BK_{s}^{(T)})+\sum_{s=1}^{t}\prod_{j=s}^{t-1}(A+BK_{j}^{(T)})BL_{s-1}^{(T)}\Sigma\nonumber,
 \end{flalign}
 where $\Sigma= E[x_{0}^{1}(x_{0}^{2})^{T}]( E[x_{0}^{2}(x_{0}^{2})^{T}])^{-1}$, $P_{T}^{(T)}=0$, $G_{0}^{(T)}=I$. Moreover, the optimal cost is
 \begin{flalign}\label{eq:optcost}
J_{T}(\underline{\gamma}^{*}_{T})&=\frac{2}{T}\bigg\{ E\bigg[(x_{0}^{1})^{T}P_{0}^{(T)}x_{0}^{1}\bigg]+\sum_{t=0}^{T-1} E\bigg[(w_{t}^{1})^{T}P_{t}^{(T)}w_{t}^{1}\bigg]\\
&+\sum_{t=0}^{T-1} E\bigg[\bigg( E[x_{0}^{2}|x_{0}^{1}]\bigg)^{T}\bigg((L_{t}^{(T)})^{T}B^{T}P_{t+1}^{(T)}BL_{t}^{(T)}\bigg) E[x_{0}^{2}|x_{0}^{1}]\bigg]\nonumber\\
&+\sum_{t=1}^{T-1} E\bigg[(x_{0}^{1})^{T}(A^{T})^{t}P_{t+1}^{(T)}BL_{t}^{(T)} E[x_{0}^{2}|x_{0}^{1}]\bigg]\bigg\}.\nonumber
\end{flalign}
 \end{theorem}
 \begin{proof}
 Following from \cite{HoChu} and Radner's theorem \cite{Radner}, person-by-person optimality implies global optimality due to the uniqueness of the person-by-person optimal policy. That is because the information structure is partially nested, and LQG dynamic teams can be reduced to a static one using Ho-Chu's static reduction \cite{HoChu}. Hence, we only need to show that the policy satisfying \eqref{eq:4.9} and \eqref{eq:4.10} is person-by-person optimal. We show that for DM$^{1}$, $J(\underline\gamma^{*}_{T},\underline\gamma^{*}_{T})\leq J((\underline\gamma^{-t*}_{T},\beta),\underline\gamma^{*}_{T})$ for all $\beta \in \Gamma^{t}$ where $(\underline\gamma^{-t*}_{T},\beta)=(\gamma_{0:t-1}^{*},\beta,\gamma_{t+1:T-1}^{*})$. This implies that $(\underline\gamma^{*}_{T},\underline\gamma^{*}_{T})$ is person-by-person optimal thanks to Corollary \ref{lem:4.4.4.4} since the dynamic team is symmetrically optimal (by exchanging policies $(\underline\gamma^{-t*}_{T},\beta)$ with $\underline\gamma^{*}_{T}$ which implies $J(\underline\gamma^{*}_{T},\underline\gamma^{*}_{T})\leq J(\underline\gamma^{*}_{T},(\underline\gamma^{-t*}_{T},\beta))$ for all $\beta \in \Gamma^{t}$ and this implies that $(\underline\gamma^{*}_{T},\underline\gamma^{*}_{T})$ is the fixed point of the equation). The proof is completed by induction. Due to space constraints, we have removed the calculation. 
 \end{proof}
 
 \begin{remark}
 The optimal policies \eqref{eq:4.9} and \eqref{eq:4.10} contain two parts which can be interpreted as follows: the first part, $k_{t}^{(T)}x_{t}^{i}$, is equivalent to the optimal policy of the branch (DM) by ignoring the other branch in the optimization problem (in this case, this is equivalent to the centralized policies since the information structure of each branch (DM) is centralized). The second part corresponds to the correlation term between branches (DMs).
 \end{remark}

 In the following, we generalize the result of Theorem \ref{the:4.2} to $N$-DM LQG dynamic teams.
 Assume that the dynamics for $i=1,2,...,N$ are defined as \eqref{eq:state}.
\begin{itemize}[wide = 0pt]
\item[\textbf{Problem} \bf{($\mathcal{P}_{T}^{N, \text{tree}}$)}:]
Consider the expected cost function of $\underline{\gamma}_{T}^{1:N}$ as 
\begin{flalign}
\scalemath{1}{J_{T}^{N}(\underline{\gamma}_{T}^{1:N})=\frac{1}{T}\sum_{t=0}^{T-1}\sum_{i=1}^{N} E^{\underline{\gamma}_{T}^{1:N}}\bigg[(x_{t}^{i})^{T}Qx_{t}^{i}+(u_{t}^{i})^{T}Ru_{t}^{i} +\sum_{j=1,j\not=i}^{N}(u_{t}^{i})^{T}\tilde Ru_{t}^{j}+(u_{t}^{j})^{T}\tilde Ru_{t}^{i}\bigg]}\label{eq:4.12},
\end{flalign}
where $\underline{\gamma}_{T}^{i}=\gamma^{i}_{0:T-1}$ for $i=1,\dots,N$ and $R,\tilde R>0$ and $Q \geq 0$. Let $I_{t}^{i}=\{x_{[0:t]}^{i}, u_{[0:t-1]}^{i}\}$.
\end{itemize}
\begin{corollary}\label{the:3.10}
 For a fixed $T$ and $N$, consider a finite horizon team problem defined as ($\mathcal{P}_{T}^{N,\text{tree}}$). If $(x_{0}^{1:N})$ are exchangeable with an identical zero mean Gaussian distribution, and $w^{i}_{t}$s for $i=1,\dots,N$ are i.i.d. zero mean Gaussian random vectors, independent for $t=0,\dots,T-1$, and independent of initial states, then
 \begin{flalign*}
 &\scalemath{1}{u^{i*,(T),(N)}_{t}=K_{t}^{(T)}x_{t}^{i}+L_{t}^{(N),(T)}\sum_{j=1,j\not=i}^{N} E[x_{0}^{j}|x_{0}^{i}]},
 \end{flalign*}
 where $K_{t}^{(T)}$ and $P_{t}^{(T)}$ satisfy \eqref{eq:3.4.15} and \eqref{eq:3.4.16}, and $L_{t}^{(N),(T)}$ is a function of $K_{0:t}^{(T)}$. 
\end{corollary}
\begin{proof}
The proof is similar to the one of Theorem \ref{the:4.2}.
\end{proof}

Now we consider more general setup where using Corollary \ref{lem:4.4.4.4}, we establish a structural result for the case where the information structure of each decision maker over time satisfies a structure which is identical for all DMs and is partially nested. An example of such a graph structure has been depicted in Fig. 2. 

{\begin{itemize}[wide = 0pt]
\item[\textbf{Problem} \bf{($\mathcal{P}_{T}^{N}$)}:]
Consider a finite horizon expected cost of $\underline{\gamma}_{T}^{1:N}$ as \eqref{eq:4.12} with the information structure $I_{t}^{i}=\{{y}^{i}_{t}, {y}_{\downarrow t}^{i}\}$ where ${y}^{i}_{t}$ is defined in \eqref{eq:observationlqg} and dynamics is defined in \eqref{eq:state}.
\end{itemize}}
\begin{theorem}
 For a fixed $T$ and $N$, consider a finite horizon team problem defined as ($\mathcal{P}_{T}^{N}$). If $(x_{0}^{1}, \dots, x_{0}^{N})$ are exchangeable with an identical zero mean Gaussian distribution and $w^{i}_{t}$s for $i=1,\dots, N$ are i.i.d. zero mean Gaussian random vectors, independent for all $t=0,\dots, T-1$  and independent of initial states, then
\begin{flalign*}
 &\scalemath{1}{u^{i*,(T),(N)}_{t}=K_{t}^{(T)}{y}_{t}^{i}+L_{t}^{(N),(T)}\sum_{j=1,j\not=i}^{N} E[x_{0}^{j}|y_{t}^{i}]},
 \end{flalign*}
 where $K_{t}^{(T)}$ are obtained by considering only one DM and ignoring other DMs.
\end{theorem}
\begin{proof}
The proof is similar to that of Theorem \ref{the:4.2} by \cite{HoChu} and Corollary \ref{lem:4.4.4.4}.
\end{proof}
\begin{remark}
A related work is \cite{nayyar2014optimal}, where structural results for optimal policy have been obtained for finite horizon LQG problems on graphs. In our analysis above, the  structural result for the optimal policy is obtained without assuming that decision makers who have no common ancestors and no common descendants have either uncorrelated noise or are decoupled through the cost function. Instead, exchangeable partially nested LQG teams with correlated initial states and disturbances are considered.  Moreover, here, the graph structures may not be trees in general, as opposed to \cite{nayyar2014optimal} where a multi-tree structure has been imposed on a graph.
\end{remark}

{Now, we present results for LQG teams with a mean-field coupling through the cost function. First, using Corollary \ref{the:3.10}, we obtain globally optimal policies for $N$-DM teams with a mean-field coupling and correlated initial states and disturbances. Next, as an implication of Theorem \ref{the:mft}, we show the convergence of optimal policies for LQG $N$-DM mean-field teams on a tree to the corresponding optimal policy of mean-field teams.}
 Let $I_{t}^{i}=\{x_{[0:t]}^{i}, u_{[0:t-1]}^{i}\}$  for $i \in \mathbb{N}$, and dynamics be as \eqref{eq:state}.
\begin{itemize}[wide = 0pt]
\item[\textbf{Problem} \bf{($\mathcal{P}_{T,\text{LQG}}^{N, \text{MF}}$)}:] Consider the expected cost function of $\underline{\gamma}_{T}^{N}$ as 
\begin{flalign}\label{eq:4.4.12}
&\scalemath{1}{J_{T}^{N}(\underline{\gamma}_{T}^{N})=\frac{1}{T}\sum_{t=0}^{T-1}\sum_{i=1}^{N} E^{\underline{\gamma}_{T}^{1:N}}\bigg[(x_{t}^{i})^{T}Qx_{t}^{i}+(u_{t}^{i})^{T}Ru_{t}^{i}}\\
&~~~~~~~~~~~\scalemath{1}{+\frac{1}{N-1}\sum_{j=1,j\not=i}^{N}(u_{t}^{i})^{T}\tilde Ru_{t}^{j}+(u_{t}^{j})^{T}\tilde Ru_{t}^{i}+\frac{1}{N-1}\sum_{j=1,j\not=i}^{N}(x_{t}^{i})^{T}\tilde Qx_{t}^{j}+(x_{t}^{j})^{T}\tilde Qx_{t}^{i}\bigg]}\nonumber,
\end{flalign}
where $R,\tilde R>0$ and $Q, \tilde Q \geq 0$.
\item[\textbf{Problem} \bf{($\mathcal{P}_{T,\text{LQG}}^{\infty,\text{MF}}$)}:] Consider the expected cost function of $\underline{\gamma}_{T}$ as
\begin{flalign}\label{eq:lqgmfinf}
&\scalemath{1}{J^{\infty}_{T}(\underline{\gamma}_{T})=\limsup\limits_{N\rightarrow \infty}J_{T}^{N}(\underline{\gamma}_{T})}.
\end{flalign}

\end{itemize}
\begin{corollary}\label{cor:3.9}
 For a fixed $T$ and $N$, consider a finite horizon team problem defined as ($\mathcal{P}_{T,\text{LQG}}^{N,\text{MF}}$). If $(x_{0}^{1:N})$ are exchangeable zero mean Gaussian random vectors, and $w^{i}_{t}$s are i.i.d. zero mean Gaussian random vectors for $i=1,\dots,N$, independent for $t=0,\dots, T-1$, and independent of initial states, then
  \begin{flalign}\label{eq:optmf}
 &\scalemath{1}{u^{i*,(T),(N)}_{t}=K_{t}^{(T)}x_{t}^{i}+\frac{L_{t}^{(N),(T)}}{N-1}\sum_{j=1,j\not=i}^{N} E[x_{0}^{j}|x_{0}^{i}]},
 \end{flalign}
 where $K_{t}^{(T)}$ and $P_{t}^{(T)}$ satisfy \eqref{eq:3.4.15} and \eqref{eq:3.4.16}, and $L_{t}^{(N),(T)}$ is a function of $K_{0:t}^{(T)}$.
\end{corollary}
\begin{proof}
The proof is similar to the one in Theorem \ref{the:4.2}.
\end{proof}
\begin{corollary}
 For a fixed $T$, consider a finite horizon team problem defined as ($\mathcal{P}_{T,\text{LQG}}^{\infty,\text{MF}}$). Assume $\{x_{0}^{i}\}_{i \in \mathbb{N}}$ are exchangeable random vectors with zero mean Gaussian distribution, and $w^{i}_{t}$s are i.i.d. zero mean Gaussian random vectors for $i \in \mathbb{N}$, independent for $t=0,\dots, T-1$, and independent of initial states. If $L_{t}^{(N),(T)}$ in \eqref{eq:optmf} converges pointwise as $N \to \infty$ to $L_{t}^{(\infty),(T)}$, then
  \begin{flalign*}
 &u^{i*,(T),(\infty)}_{t}=K_{t}^{(T)}x_{t}^{i}+L_{t}^{(\infty),(T)}\Sigma x_{0}^{i},
 \end{flalign*}
 where $K_{t}^{(T)}$ and $P_{t}^{(T)}$ satisfy \eqref{eq:3.4.15} and \eqref{eq:3.4.16}, and $\Sigma= E[x_{0}^{1}(x_{0}^{2})^{T}]( E[x_{0}^{2}(x_{0}^{2})^{T}])^{-1}$.
\end{corollary}
\begin{proof}
Following from \cite[page 9]{aldous2006ecole}, since $\{x_{0}^{i}\}_{i \in \mathbb{N}}$ are exchangeable Gaussian random vectors, we can describe them explicitly as $x_{0}^{i}=\omega_{0}+\theta^{i}$ where $(\theta^{1},\theta^{2}, \dots)$ are i.i.d. mean zero Gaussian and independent of mean zero Gaussian random vector $\omega_{0}$. Now, we invoke Theorem \ref{the:mft} (or Theorem \ref{the:mftergodic}) and Corollary \ref{cor:3.9} and we use Remark \ref{rem:mft} to complete the proof.
\end{proof}
\subsection{Average cost infinite horizons problems for partially nested dynamic teams}\label{sec:ave}
In the following, we consider average cost problems with a symmetric partially nested information structure. We note that the optimality of linear policies for infinite horizon LQG problems is an open problem in its generality. In this subsection, we provide a positive result for a class of such problems.

 Now, we consider an infinite horizon team problem and we use the result in Section \ref{sec:3a}.
\begin{itemize}[wide = 0pt]
 \item[\textbf{Problem} \bf{($\mathcal{P}_{\infty}^{\text{tree}}$)}:]
 Consider the expected cost function of $(\underline{\gamma}^{1},\underline{\gamma}^{2})$ as
  \begin{flalign}\label{eq:1.2}
&J(\underline{\gamma}^{1:2})=\limsup\limits_{T\rightarrow \infty} E^{(\underline{\gamma}^{1},\underline{\gamma}^{2})}[c(x_{0:T-1}^{1:2},u_{0:T-1}^{1:2})],
\end{flalign}
where the cost function is defined as \eqref{eq:1.1} and $I_{t}^{i}=\{x_{[0:t]}^{i}, u_{[0:t-1]}^{i}\}$.
\end{itemize}

 First, we introduce a lemma essential for Theorem \ref{the:4.}.
\begin{lemma}\label{lem:1}
Consider the sequence $\{a^{i}_{T}\}_{i=1}^{T}$. Assume $\lim\limits_{T\rightarrow \infty}a_{T}^{i}=a$ for $i=0,\dots,T-1$. If for every fixed $T\in \mathbb{N}$, $a_{T}^{l}=a_{T+1}^{l+1}$ for all $l=0\dots,T-1$, then $\lim\limits_{T\rightarrow \infty}\frac{1}{T}\sum_{i=1}^{T}a_{T}^{i}=a$.
\end{lemma}
\begin{proof}
We have
\begin{equation*}
\scalemath{1}{\lim\limits_{T\rightarrow \infty}\frac{1}{T}\sum_{i=1}^{T}a_{T}^{i}=\lim\limits_{T\rightarrow \infty}\frac{1}{T}\sum_{i=0}^{T-1}a_{T-i}^{1}=\lim\limits_{T\rightarrow \infty}\frac{1}{T}\sum_{k=1}^{T}a_{k}^{1}=a},
\end{equation*}
where the second equality follows from $a_{T}^{l}=a_{T+1}^{l+1}$  and the last equality follows from the Ces\'aro mean argument.
\end{proof}

\begin{theorem}\label{the:4.}
Consider average cost infinite horizon team problems defined as ($\mathcal{P}_{\infty}^{tree}$). Assume $(A,B)$ are stabilizable and $(A,Q^{\frac{1}{2}})$ are detectable. Assume $x_{0}^{1}$ and $x_{0}^{2}$  are exchangeable with an identical zero mean Gaussian distribution and $w^{i}_{t}$s are i.i.d. zero mean Gaussian random variables for $i=1,2$ and for all $t=0,\dots, T-1$ and independent of initial states. If $L_{t}^{(T)}$ in \eqref{eq:3.4.17} converges pointwise to $L_{t}^{(\infty)}$ as $T \to \infty$, then the pointwise limit of the sequence of optimal policies for ($\mathcal{P}_{T}^{\text{tree}}$) is team optimal for ($\mathcal{P}_{\infty}^{tree}$) and stabilizes the closed-loop system,
 \begin{flalign*}
 &u^{1*,(\infty)}_{t}=Kx_{t}^{1}+L_{t}^{(\infty)} E[x_{0}^{2}|x_{0}^{1}],\\
 &u^{2*,(\infty)}_{t}=Kx_{t}^{2}+L_{t}^{(\infty)} E[x_{0}^{1}|x_{0}^{2}],
 \end{flalign*}
 where $K, P, L_{t}^{(\infty)}$ and $G_{t}^{(\infty)}$ are the pointwise limit of the ones for ($\mathcal{P}_{T}^{tree}$) as $T \to \infty$.
\end{theorem}
\begin{proof}
We show $\limsup\limits_{T\rightarrow \infty}J_{T}(\underline{\gamma}_{T}^{*})=J(\underline{\gamma}_{\infty}^{*})$ and invoke \cite[Theorem 5]{sanjari2018optimal} or \cite[Theorem 1]{mahajan2013static} to complete the proof. From \eqref{eq:optcost}, we have
\begin{flalign}
&\scalemath{1}{\lim\limits_{T\rightarrow \infty}|J_{T}(\underline{\gamma}_{\infty}^{*})-J_{T}(\underline{\gamma}_{T}^{*})|}\nonumber\\
&\scalemath{1}{=\lim\limits_{T\rightarrow \infty}\frac{2}{T}\bigg| E\bigg[(x_{0}^{1})^{T}(P_{0}^{(T)}-P)x_{0}^{1}\bigg]+\frac{2}{T}\sum_{t=0}^{T-1} E\bigg[(w_{t}^{1})^{T}(P_{t}^{(T)}-P)w_{t}^{1}\bigg]}\label{eq:a.35}\\
&\scalemath{1}{+\frac{2}{T}\sum_{t=0}^{T-1} E\bigg[\bigg( E[x_{0}^{2}|x_{0}^{1}]\bigg)^{T}\bigg((L_{t}^{(T)})^{T}B^{T}P_{t+1}^{(T)}BL_{t}^{(T)}-(L_{t}^{(\infty)})^{T}B^{T}PBL_{t}^{(\infty)}\bigg) E[x_{0}^{2}|x_{0}^{1}]\bigg]}\label{eq:a.36}\\
&\scalemath{1}{+\frac{2}{T}\sum_{t=0}^{T-1} E\bigg[(x_{0}^{1})^{T}(A^{T})^{t-1}\bigg(P_{t+1}^{(T)}BL_{t}^{(T)}-PBL_{t}^{(\infty)}\bigg) E[x_{0}^{2}|x_{0}^{1}]\bigg]\bigg|}\label{eq:a.37}\\
&\scalemath{1}{=0}\label{eq:a.38},
\end{flalign}
where \eqref{eq:a.35} is zero since $P_{0}^{(T)}$ converges to $P$ using Lemma \ref{lem:1}  since $P_{t+1}^{(T+1)}=P_{t}^{(T)}$. Expression \eqref{eq:a.37} converges to zero since $L_{t}^{(T)}$ in \eqref{eq:3.4.17} converges pointwise to $L_{t}^{(\infty)}$ as $T \to \infty$, we have $\sum_{s=t+1}^{\infty}B^{T}(A^{T})^{s-t}PBL_{s}^{(\infty)}<\infty$, and this implies that $\lim\limits_{s\rightarrow \infty}L_{s}^{(\infty)}=0$. Hence, we have for every $\epsilon>0$, there exists $\hat{N}>T$ such that for every $t>\hat{N}$, $|Tr[L_{t}^{(\infty)}(L_{t}^{(\infty)})^{T}]|<\epsilon$. We define $L_{t}^{(T)}=0$ for $t>T$. Expression \eqref{eq:a.36} is equal to zero following from Lemma \ref{lem:1} and the fact that $|Tr[L_{t}^{(\infty)}(L_{t}^{(\infty)})^{T}]|<\epsilon$  for every $t>\hat{N}$. Hence, equality \eqref{eq:a.38} is true and global optimality follows from \cite[Theorem 5]{sanjari2018optimal}. The closed loop system is stable since $\limsup\limits_{t\rightarrow \infty} E(||x_{t}^{1}||^{2})<\infty$  following from $||A+BK||<1$ (all the eigenvalues of $A+BK$ are inside of the unit circle), and since  $||L_{t}^{(\infty)}||$ is uniformly bounded.
\end{proof}

\section{Conclusion}\label{sec:con}
In this paper, we studied dynamic teams with symmetric information structures. We presented a characterization for symmetrically optimal teams for convex exchangeable team problems. For mean-field teams with symmetric information structures, we showed the convergence of optimal policies for mean-field teams with $N$ decision makers to the corresponding optimal policy of mean-field teams. We obtained globally optimal solutions for LQG dynamic team problems with symmetric partially nested information structures. Moreover, we obtained globally optimal policies for average cost infinite horizon problems of LQG dynamic teams. 

In this paper since we worked under the convexity assumption, the information structure does not allow for the mean-field coupling in the dynamics. In our recent work \cite{SSYdefinetti2020}, we relaxed the convexity assumption to arrive at complementary existence and structural results.

\bibliographystyle{plain}

\end{document}